\documentclass[12pt,reqno,letter]{article}
\usepackage[margin=1in]{geometry}

\usepackage{listings}
 \usepackage[usenames,dvipsnames]{color}
 \usepackage[colorlinks=true, pdfstartview=FitV, linkcolor=blue, citecolor=blue, urlcolor=blue]{hyperref}

\usepackage[nottoc,numbib]{tocbibind}
\usepackage{helvet}

\usepackage{amsmath,amsthm,amsfonts,amssymb,amscd}

\usepackage{mathtools} % Mathe
\usepackage{mathabx}
\usepackage{float}
\usepackage{graphicx}
\usepackage{makeidx}
\usepackage{enumerate}
\usepackage{mathtools}
\usepackage{xfrac}

\usepackage[mathscr]{euscript}
\usepackage{stmaryrd}
\usepackage{microtype}
\usepackage{booktabs}
\usepackage{cleveref}
\usepackage{bookmark}
\usepackage{mathrsfs}

\usepackage{emptypage}
\usepackage{tikz}

\DeclareMathAlphabet{\mathpzc}{OT1}{pzc}{m}{it}

\theoremstyle{plain}
\newtheorem{theorem}{\scshape Theorem}[section]

\newtheorem{lemma}[theorem]{\scshape Lemma}
\newtheorem{corollary}[theorem]{\scshape Corollary}
\theoremstyle{definition}
\newtheorem{definition}[theorem]{\scshape Definition}

\newtheorem{remark}[theorem]{\scshape Remark}

\def\<{\langle}
\def\>{\rangle}

\def\F{\mathcal{F}}

\def\H{{:\!\mathcal H\!:}}
\def\N{\mathbb{N}}
\def\P{\mathbb{P}}
\def\R{\mathbb{R}}
\def\T{\mathbb{T}}
\def\Z{\mathbb{Z}}

\def\1{\mathbf{1}}

\def\inte#1{
\displaystyle\mathop{#1\kern0pt}^\circ
}

\newcommand{\beq}{\begin{equation}}
\newcommand{\eeq}{\end{equation}}
\newcommand{\ben}{\begin{eqnarray}}
\newcommand{\een}{\end{eqnarray}}
\newcommand{\beno}{\begin{eqnarray*}}
\newcommand{\eeno}{\end{eqnarray*}}

\def\T{\mathbb{T}}

\def\virgp{\raise 2pt\hbox{,}}
\def\cdotpv{\raise 2pt\hbox{;}}

\def\C{\mathop{\mathbb C\kern 0pt}\nolimits}
\def\DD{\mathop{\mathbb D\kern 0pt}\nolimits}
\def\EE{\mathop{\mathbb E\kern 0pt}\nolimits}
\def\K{\mathop{\mathbb K\kern 0pt}\nolimits}
\def\N{\mathop{\mathbb  N\kern 0pt}\nolimits}
\def\Q{\mathop{\mathbb  Q\kern 0pt}\nolimits}
\def\R{{\mathop{\mathbb R\kern 0pt}\nolimits}}
\def\SS{\mathop{\mathbb  S\kern 0pt}\nolimits}
\def\St{\mathop{\mathbb  S\kern 0pt}\nolimits}
\def\Z{\mathop{\mathbb  Z\kern 0pt}\nolimits}
\def\ZZ{{\mathop{\mathbb  Z\kern 0pt}\nolimits}}
\def\H{{\mathop{{\mathbb  H\kern 0pt}}\nolimits}}
\def\PP{\mathop{\mathbb P\kern 0pt}\nolimits}
\def\TT{\mathop{\mathbb T\kern 0pt}\nolimits}

\newcommand{\with}{\quad\hbox{with}\quad}
\usepackage{geometry}

\begin{document}

\title{On White Noise Solutions of mSQG Equations on $\R^2$
 }
\maketitle
%\subtitle{Using  the  LaTex Template}

\author{Siyu Liang  \footnotemark[1] \footnotemark[2] \footnotemark[3]}
        \footnotetext[1]{Department of Mathematics, University of Bielefeld, D-33615 Bielefeld, Germany,
                sliang@math.uni-bielefeld.de}
 \footnotetext[2]{Academy of Mathematics and Systems Science, Chinese Academy of Sciences, Beijing 100190, China}
        \footnotetext[3]{School of Mathamatical Sciences, University of Chinese Academy of Sciences, Beijing 100049, China}

\begin{abstract}
In this paper, we show  existence of   white noise solutions for  weak formulations of  modified Surface  Quasi-Geostrophic (mSQG) equations.  Based on previous  results (\cite{FS}) on white noise solutions for mSQG equations  on
the torus $\T^2$,  we show a similar result for the whole space $\R^2$ by letting the volume of the torus go to infinity and applying compactness methods (Skorokhod's theorem).  

\end{abstract}


\section{Introduction}

 In this paper,  we  study the stationary solutions of the
   following
 modified Surface  Quasi-Geostrophic 
 equations (mSQG equations) 
on the torus $\T^2:=(\mathbb{R}/2\pi \mathbb{Z})^2$  and the whole space $\R^2$
 \begin{equation*}\tag{mSQG}\label{2DmSQG}
  \left\{ \aligned
  \partial_t \omega+ u\cdot\nabla \omega &= 0, \\
  u &= \nabla^{\perp} (-\Delta)^{-(1+\epsilon)/2} \omega,
  \endaligned
  \right.
  \end{equation*}
  where $0<\epsilon<1$,  $\nabla^{\perp}=(-\partial_2, \partial_1 )
  $, and  $(-\Delta)^{-(1+\epsilon)/2}$ is the fractional Laplacian operator,     the definition of which is given  in  Section \ref{subsection12}.

  When $\epsilon=1$,  the above equation becomes the Euler equation, and for 
  $\epsilon=0$ it is called the   Surface Quasi-Geostrophic  (SQG) equation.

   The SQG equations   are approximations to the shallow water equations with a small Rossby number (which goes to $0$ in the limit),  where a  small Rossby number
  means the system is mainly determined by the Coriolis force which is caused by earth rotation.   It is also called  ``(nearly) in geostrophic balance".

  The SQG equation is obtained from  the $3D$ Quasi-Geostrophic equation by assuming the potential vorticity to be  identically $0$.  
  The SQG equation ($\epsilon=0$) is  introduced in 
  \cite{ConMajTab1994},  where a striking mathematical and physical analogy is developed between the structure and formation of singular solutions of 
SQG equations and the potential formation of finite-time singular solutions for the $3D$ Euler equations.  For a more physical background of Quasi-Geostrophic  equations and the  formulation of SQG equations we refer to \cite{HPGS,lapeyre2017surface,DAVIES20031787,ped} and \cite{vallis2017atmospheric}.

   The classical incompressible Euler equations are well-known 
  and  have been studied extensively in the literature, see for example,
  \cite{LANDAU1987x},  \cite{Chemin} and \cite{Lions}. 
  It is constructed in   \cite{FS}  a white noise  solution of Euler equations
    by the following point-vortex system:\\
  \[
\omega_{t}^{N}=\frac{1}{\sqrt{N}}\sum_{n=1}^{N}\xi_{n}\delta_{X_{t}^{n,N}},
\]
and
 for every $N\in\mathbb{N}$,  the finite dimensional dynamics \begin{equation*}
\frac{dX_{t}^{i,N}}{dt}=\sum_{j=1}^{N}\frac{1}{\sqrt{N}}\xi_{j}K\left(
X_{t}^{i,N}-X_{t}^{j,N}\right)  \qquad i=1,...,N\label{vortex system}%
\end{equation*}
in
$\left(  \mathbb{T}^{2}\right)  ^{N}$
with initial condition $\left(  X_{0}^{1,N},...,X_{0}^{N,N}\right)  \in\left(
\mathbb{T}^{2}\right)  ^{N}\backslash\Delta_{N}$, $
\Delta_N:=\{(x_1, x_2,...,x_N)\in (\mathbb{T}^{2})^{N} ;x_i=x_j \text{ for some } i\neq j,i,j=1,2,...N \}
$,   where  $K$ is the
Biot-Savart kernel on $\mathbb{T}^{2}$ (we set $K\left(  0\right)  =0$), and the
intensities $\xi_{1},...,\xi_{N}$ are (random) numbers of any sign.

Exploiting the similarity to the Euler equations, 
many classical results have also been obtained for   SQG and mSQG equations.  For example, global existence  of weak solutions to  SQG equations is known in the spaces $L^p(\R^2)$,  for $p \in (4/3,\infty)$  (see \cite{Res95, Marchand2008}).   In \cite{bsv16},  non-uniqueness of weak solutions is proved in a certain class  by using the methods of  convex integration.    mSQG equations, which are equations between SQG equations and Euler equations, have also been studied by many papers,  such as
\cite{Hunter2019,  Hunter2020GlobalSF, ChaConCorGanWu2012,  Geldhauser_Romito_2020}. In a recent work
\cite{arXiv210411048},      
    nontrivial global (classical) solutions of the mSQG equations have been constructed.

 Similarly to Euler equations, there are also some results via  point-vortex model to approximate mSQG equations, such as \cite{FS}, \cite{Luo2019RegularizationBN},     \cite{LUO2021236},    \cite{Garcia_2020}(for more general models),
  \cite{Geldhauser_Romito_2020}, and \cite{rosenzweig}. 
  In \cite{FS},   the point-vortex approximation is used  
  to show the existence of  white noise solutions of 
  the weak formulation of mSQG equations on the torus (see Definition \ref{whitenoisesolution}   for
 the definition of white noise solutions).

  However, for the case of $\R^2$, there is no  result 
  of existence of white noise solutions 
  as far as we know. 
  In this paper,   we will  generalize the result of the existence of white noise solutions of mSQG equations  to  $\R^2$.
  But we will prove it in a different way.
   Since there have been previous results of the existence of white noise solutions  on $\T^2$ (\cite{FS}),  we do not use vortex systems to approximate solutions. 
   Instead, since the existence of 
   white noise solutions holds on the torus of any volume,
    we will let the volume of torus go to infinity and apply the compactness methods.

    The reason that we consider the mSQG equations on $\R^2$ is  that the kernel corresponding to  $\nabla^{\perp} (-\Delta)^{-(1+\epsilon)/2}$   is dominated by 
  $C\frac{1}{|x|^{2-\epsilon}}$.
  Therefore, the kernel corresponding to mSQG equations ($0<\epsilon<1$) has a better behaviour at infinity compared to $2D$ Biot-Savart kernel ($\epsilon=1$).  When $\epsilon=0$, the behaviour of the kernel at infinity is even better. However, its behaviour at the origin is bad. Therefore, 
  in the case of SQG equations ($\epsilon=0$),
  it is difficult  to obtain even the existence of white noise solutions    
    on the torus $\T^2$.

    \textbf{Main results of this paper }:  we prove the existence of 
    white noise solutions to the weak formulations of \eqref{2DmSQG}  on $\R^2$
  (Theorem \ref{maintheorem}).

  In Section \ref{sec2}, we first introduce some function and distribution spaces. 
  Then we show the properties and the relations of the kernel of \eqref{2DmSQG} on $\T^2$ and $\R^2$. 
  Finally we introduce the definition of space white noise. 
 In Section \ref{sec3}, we define the nonlinear term by approximating sequences. In Section \ref{sec4}, we prove our main result.

  \section{Preliminaries}\label{sec2}

  \subsection{Function and distribution spaces}

  In this section we introduce some function and distribution spaces.
  \subsubsection{Function and distribution spaces on $\mathbb{R}^{2}$}\label{cp1sec1subsec1}
Denote by $\mathcal{S}(\mathbb{R}^2)$ the Schwartz space and 
$\mathcal{S}'(\mathbb{R}^2)$ its dual space.
Denote by $C_c^{\infty}(\mathbb{R}^2)$ the space of  smooth functions
on $\mathbb{R}^{2}$  with compact support. 
Denote by $C_c^{k}(\mathbb{R}^2)$ the space of  compact supported functions on $\mathbb{R}^{2}$ which have $k$th
continuous derivatives. \\
On $\mathbb{R}^{2}$, we recall the classical (non-homogeneous) Sobolev spaces:
\beq\label{defhs}
H^{s}(\mathbb{R}^{2}):=\Bigl\{u\in \mathcal{S}'(\mathbb{R}^{2});\| u \|_{H^{s}(\mathbb{R}^{2})}^{2}:=\int_{\mathbb{R}^{2}}(1+| \xi |^{2})^s\mid \hat{u}(\xi)\mid^{2}d\xi<\infty\,\Bigr\},   
\eeq
where $s\in\R$, 
 and 
 $$\hat{u}(\xi)=\mathcal{F}u(\xi):=\int_{\R^2}u(x)e^{- ix\cdot \xi }dx,
   $$
  denotes the Fourier transform of $u$ on $\R^2$.
One knows that   $H^{s}(\mathbb{R}^{2})$ is a Hilbert space with $H^{-s}(\mathbb{R}^{2})$ as its dual space.

For $s\in\R$,   we define the spaces of vector fields $H^{s}(\mathbb{R}^{2};\mathbb{R}^2)$ to be the sets of the vector-valued functions with both components in  $H^{s}(\mathbb{R}^{2})$.
 For simplicity, from now on,   we will use the same notations of vector fields and function spaces when there is no confusion.

We introduce the following weighted Sobolev norms and spaces. 
\begin{definition}[Weighted Sobolev norms and spaces]
{\sl Let $\rho\in L^1_{loc}(\R^d)$ and $\rho(x)\geq 0$.
Define the weighted Sobolev norms 
$\|\cdot \|_{H^{s}(\mathbb{R}^{d},\rho)}$ by
$$\|\cdot \|_{H^{s}(\mathbb{R}^{d},\rho)}:=\|\rho \cdot\|_{H^s(\mathbb{R}^{d})}.$$
Define the weighted Sobolev spaces $H^{s}(\mathbb{R}^{d},\rho)$ as the subspace of 
$\mathcal{S}'(\mathbb{R}^d)$ 
such that  $\|\cdot \|_{H^{s}(\mathbb{R}^{d},\rho)}$ finite.}
\end{definition}

\noindent
Since  we always consider the $2D$ case,
from now on for simplicity we use the notation $H^{s}(\rho)$
instead of $H^{s}(\mathbb{R}^{2},\rho)$ when no confusion occurs.
Moreover, we define the space $H^{-1-}(\rho)$ as  the space
 $\bigcap\limits_{\epsilon>0}H^{-1-\epsilon}(\rho)$ with the following Frechet metric $d$:
 $$d(u,v)=\sum\limits_{n=1}^{\infty}\frac{1}{2^n}
\frac{\|u-v\|_{H^{-1-\frac{1}{n}}(\rho)}}{1+\|u-v\|_{H^{-1-\frac{1}{n}}(\rho)}}.
 $$
Then,  convergence in $H^{-1-}(\rho)$  is equivalent to convergence in 
 $H^{-1-\epsilon}(\rho)$  for each $\epsilon>0$.

 \noindent
 Let 
 $$\rho_{\sigma}(x):= \frac{1}{\langle x\rangle^{\sigma}}    $$
and 
 $$\rho_{\sigma'}(x):= \frac{1}{\langle x\rangle^{\sigma'}},    $$
 where $\langle x\rangle=(1+|x|^2)^{\frac{1}{2}}$.\\
 
 \noindent
The following lemma is proved in \cite[Theorem 6.31]{bookHans}:
 \begin{lemma}
 {\sl  
For $0<\sigma'< \sigma$, and $s'>s$,  the distributional space 
 $H^{s'}(\rho_{\sigma'})$ 
 is compactly embedded in 
 $H^{s}(\rho_{\sigma})$.
 }
 \end{lemma}

\subsubsection{Function and  distribution spaces on $\mathbb{T}^{2}$}

Denote by $C^{\infty}(\T^2)$ 
the space of smooth functions on $\T^2$.
Noting that $\{\frac{1}{2\pi}e^{ ik\cdot x}\}_{k\in \Z^2}$ is the orthonormal basis of    $L^{2}(\T^2; \C)$,
for $u\in L^{2}(\T^2)$, we consider the Fourier expansion of $u$:
 $$u(x)=\sum\limits_{k\in \mathbb{Z}^{2} }\hat{u}_{k}\frac{1}{2\pi}e^{ ik\cdot x} \with
 \hat{u}_{k}=\overline{\hat{u}_{-k}},$$
 where $\hat{u}_{k}:=\frac{1}{2\pi}\int_{\mathbb{T}^{2}}u(x)e^{- ik\cdot x}dx$ denotes the $k$th Fourier coefficient of $u$ on $\mathbb{T}^{2}$. It follows from Fourier-Plancherel equality that the above series is convergent  in $L^{2}(\mathbb{T}^{2})$.
 Define the Sobolev norm for $s\in\R$ :
 \beq\label{hstorus}
 \| u \|_{H^{s}(\mathbb{T}^{2})}^{2}:=\sum\limits_{k\in \mathbb{Z}^{2} }(1+\mid k \mid^{2} )^s\mid \hat{u}_{k}\mid^{2}.
 \eeq
 We define the Sobolev spaces $H^{s}(\mathbb{T}^{2})$   as the completion of $C^{\infty}(\mathbb{T}^{2})$ with respect to the norm $\parallel \cdot \parallel_{H^{s}(\mathbb{T}^{2})}$.
 For $s\in\R$,   we define the space of vector fields $H^{s}(\mathbb{T}^{2};\mathbb{R}^2)$ to consist of the vectors with both components in  $H^{s}(\mathbb{T}^{2})$.\\

%  \subsubsection{ Some other spaces}\label{2.1.3}

\noindent
On  $\T^2$, 
define Fr\'echet space $H^{-1-}$ to be the linear space
$\bigcap\limits_{n\geq 1} H^{-1-\frac{1}{n}}   $
 with the distance  as follows:
$$\rho_{H^{-1-}}(x,y)=\sum\limits_{n=1}^{\infty}\frac{1}{2^n}\frac{\|x-y\|_{H^{-1-\frac{1}{n}} } }
{1+{\|x-y\|_{H^{-1-\frac{1}{n}} } }}.
$$

 \begin{remark}\label{Z2summation}
 
 \begin{enumerate}
 \
 	\item 
 In this paper, by the notation 
 $ \sum\limits_{k\in\mathbb{Z}^{2}  }
 $, we always mean
 $\lim\limits_{N\rightarrow \infty}\sum\limits_{|k_1|\leq N, |k_2|\leq N} $,
 which is particularly important when the series is  not absolutely convergent.

  \item
 From now on 
 we may suppress   the domain $(\R^2)$ or $(\T^2)$ in the notation of these function spaces,  when  no confusion occurs.
 
   \end{enumerate}
  \end{remark}

  \subsection{Introduction of weak formulations of mSQG equations  }
  \label{subsection12}
%  In this section we introduce the weak formulation form of mSQG equations.

   \subsubsection{Kernel  of \eqref{2DmSQG}  on $\R^2$}

On the whole space,  we  know that the operator    
$(-\Delta)^{-(1+\epsilon)/2}$ and 
 $\nabla^{\perp} (-\Delta)^{-(1+\epsilon)/2}$ are  defined by the Fourier multiplier $ |\xi| ^{-(1+\epsilon)}$ and
  $i \xi^{\perp} |\xi| ^{-(1+\epsilon)}$, respectively.
 Hence   if we write them in the forms of the convolution,  they are equivalent to the convolution with
    $\mathcal{\F}^{-1}\bigl( |\xi| ^{-(1+\epsilon)}\bigr)$ and
    $\mathcal{\F}^{-1}\bigl(i \xi^{\perp} |\xi| ^{-(1+\epsilon)}\bigr)$, respectively.\\
     Define
  $$K_\epsilon:=\mathcal{\F}^{-1}\bigl(i \xi^{\perp} |\xi| ^{-(1+\epsilon)}\bigr).
  $$
  Recall that on $\R^2$,   the Fourier transform and
   Fourier inverse transform  are defined as follows:
   $$\hat{f}(\xi)=\mathcal{F}f(\xi):=\int_{\R^2}f(x)e^{- ix\cdot \xi }dx,
   $$
   and
  $$\mathcal{F}^{-1}f(\xi):=\frac{1}{4\pi^2}\int_{\R^2}f(x)e^{ix\cdot \xi }dx.
   $$
 Thus we know that 
  $K_\epsilon$ is dominated by $C_{\epsilon}\frac{1}{|x|^{2-\epsilon}}$ for some constant $C_{\epsilon}$.
  The kernel is singular at the origin.

   \subsubsection{Kernel of \eqref{2DmSQG}  on the torus and the relations to the kernel on $\R^2$}\label{section222}
  
  For fixed $M$,     denote  $\mathbb{T}_M^2=(\mathbb{R}/M\mathbb{Z})^2$ to be the torus of length $M$.
   Let $f$ be a distribution in some Sobolev space $H^{-N}(\T_M^2)$,  for some $N>0$ with the Fourier expansion
  $$f(x)=\sum\limits_{k\in \mathbb{Z}^{2}\backslash \{(0,0)\}  }\hat{f}_{k}^Me_{k}^M\with
 \hat{f}_{k}^M=\overline{\hat{f}_{-k}^M},$$
 where $\hat{f}_{k}^M:=\frac{1}{M}\int_{\T_M^2}f(x)e^{-2\pi ik\cdot x/M}dx$  denotes the $k$th Fourier coefficient of $f$ on $\T_M^2$ and $e_{k}^M(x)= \frac{1}{M}e^{2\pi ik\cdot x/M}$.
  The operator $(-\Delta)^{-(1+\epsilon)/2}  $
  on the torus $ \T_M^2$
 is defined as:
  \begin{equation*}
  (-\Delta)^{-(1+\epsilon)/2}  f
 =\sum\limits_{k\in\mathbb{Z}^{2}\backslash\{(0,0)\} }
\bigl(\frac{M}{2\pi|k|}    \bigr)^{1+\epsilon}
\hat{f}_{k}^M e_{k}^M.
  \end{equation*}
  Therefore,
  \begin{equation*}
  \nabla^{\perp} (-\Delta)^{-(1+\epsilon)/2}  f
 =\sum\limits_{k\in\mathbb{Z}^{2}\backslash\{(0,0)\} }
\bigl(\frac{M}{2\pi}    \bigr)^{\epsilon}\frac{ik^{\perp}}{|k|^{1+\epsilon}}
\hat{f}_{k}^M e_{k}^M.
  \end{equation*}
  If we write it in the form of convolution, 
  \beq
  \begin{split}
  \nabla^{\perp} (-\Delta)^{-(1+\epsilon)/2}  f&=
  \sum\limits_{k\in\mathbb{Z}^{2}\backslash \{(0,0)\} }
  \bigl(\frac{M}{2\pi}    \bigr)^{\epsilon}\frac{ik^{\perp}}{|k|^{1+\epsilon}}\frac{1}{M^2}\int_{\T_M^2}f(\xi)e^{-2\pi ik\cdot \xi /M}d\xi \
   e^{2\pi ik\cdot x/M}\\
  &=\sum\limits_{k\in\mathbb{Z}^{2}\backslash \{(0,0)\} }
  \bigl(\frac{M}{2\pi}    \bigr)^{\epsilon}\frac{ik^{\perp}}{|k|^{1+\epsilon}}\frac{1}{M^2} e^{2\pi ik\cdot \cdot /M}\ast f\\
&=:K_{\epsilon}^M\ast f  ,
  \end{split}
 \eeq
 where the convolution is defined on the torus
 $\T_M^2= [-\frac{M}{2}, \frac{M}{2}]^2$.\\

  Now we want to show that similar to the case of $\R^2$, 
  $|x|^{2-\epsilon} K_{\epsilon}^M(x)$ can also be bounded by a constant
  which does not depend on $x$ and $M$.\\
  For $x\in \T_M^2$,
 $$|x|^{2-\epsilon} K_{\epsilon}^M(x)=
 \sum\limits_{k\in\mathbb{Z}^{2}\backslash \{(0,0)\} }
  \bigl(\frac{1}{2\pi}    \bigr)^{\epsilon}\frac{ik^{\perp}}{|k|^{1+\epsilon}}\bigl(\frac{|x|}{M}\bigr)^{2-\epsilon} e^{2\pi ik\cdot x /M}.
 $$
 Let $\eta=\frac{x}{M}\in [-\frac12,\frac12]^2\backslash\{(0,0)\}$,
 then 
 \beq\label{kerneltoruseta}
 |x|^{2-\epsilon} K_{\epsilon}^M(x)=
 \sum\limits_{k\in\mathbb{Z}^{2}\backslash \{(0,0)\} }
  \bigl(\frac{1}{2\pi}    \bigr)^{\epsilon}\frac{ik^{\perp}}{|k|^{1+\epsilon}}|\eta|^{2-\epsilon} e^{2\pi ik\cdot \eta}.
 \eeq
 The next lemma tells us exactly what we want.
 \begin{lemma}
 \
 {\sl 
 Define $|l|_\infty=\max\{|l_1|, |l_2|\}$. For any $\eta\in [-\frac12,\frac12]^2\backslash\{(0,0)\}$,
 \begin{enumerate}
 	\item 
 	$\lim\limits_{N\rightarrow \infty}\sum\limits_{{|l|_\infty \leq N}}
 	\mathcal{F}^{-1}(\frac{i\xi^{\perp}}{|\xi|^{1+\epsilon}})(\eta+l)
 	$
 	exists, which is denoted by $\sum\limits_{l\in\Z^2}\mathcal{F}^{-1}(\frac{i\xi^{\perp}}{|\xi|^{1+\epsilon}})(\eta+l) $,
 	and one has
 	\beq\label{summationlconvergence}
 	\big|  \sum\limits_{l\in\Z^2}\mathcal{F}^{-1}(\frac{i\xi^{\perp}}{|\xi|^{1+\epsilon}})(\eta+l)
 	\big| \lesssim |\eta|^{-2+\epsilon}.
 	\eeq
 	\item
 	$\sum\limits_{l\in\Z^2}\mathcal{F}^{-1}(\frac{i\xi^{\perp}}{|\xi|^{1+\epsilon}})(\eta+l)$ is a smooth function of $\eta$.

 \item
 It holds
 \beq\label{poissonsumformula}
  \sum\limits_{k\in\mathbb{Z}^{2}\backslash \{(0,0)\} }
  \bigl(\frac{1}{2\pi}    \bigr)^{\epsilon}\frac{ik^{\perp}}{|k|^{1+\epsilon}} e^{2\pi ik\cdot \eta}
  =\sum\limits_{l \in\Z^2}\mathcal{F}^{-1}(\frac{i\xi^{\perp}}{|\xi|^{1+\epsilon}})(\eta+l).
 \eeq
 	
 \end{enumerate}
 Then by combining 1 and 3, one obtains that \eqref{kerneltoruseta} is bounded by a constant.

 }

 \end{lemma}

 \begin{proof}
 	One knows that 
$\mathcal{F}^{-1}(\frac{i\xi^{\perp}}{|\xi|^{1+\epsilon}}) =C\nabla^{\perp}(|\eta|^{-1+\epsilon})=C\frac{\eta^{\perp}}{|\eta|^{3-\epsilon}}$, where $C$ is some constant which depends only on $\epsilon$ (see, for example, Proposition 1.29 of 
\cite{Bahouri2011Fourier}).
Set
$$\mathbb{Z}^2_{+}=\{(x_1,x_2)\in \Z^2;x_1>0\}\cup \{(0,x_2);x_2\in\N^+\},
$$
$$\mathbb{Z}^2_{-}=\{x\in \Z^2;-x\in \mathbb{Z}^2_{+} \}.
$$
Thus we have $\Z^2=\mathbb{Z}^2_{+}\cup \mathbb{Z}^2_{-}\cup\{(0,0)\}$.
Then we obtain
\beno
\begin{split}
\sum\limits_{{|l|_\infty \leq N}}
 	\frac{(\eta+l)^{\perp}}{|\eta+l|^{3-\epsilon}}&=
 	\frac{\eta^{\perp}}{|\eta|^{3-\epsilon}}+
 	\sum\limits_{\substack{(l_1,l_2)\in \mathbb{Z}^2_{+}\\ {|l|_\infty \leq N}}}
 	\big[\frac{(\eta+l)^{\perp}}{|\eta+l|^{3-\epsilon}}-\frac{(l-\eta)^{\perp}}{|l-\eta|^{3-\epsilon}}\big]
 	\end{split}
\eeno. 
Note that when $x\neq 0$, $|\nabla(\frac{x^{\perp}}{|x|^{3-\epsilon}})|\lesssim \frac{1}{|x|^{3-\epsilon}} . $
Hence we deduce
\beno
\begin{split}
\big|\sum\limits_{{|l|_\infty \leq N}}
 	\frac{(\eta+l)^{\perp}}{|\eta+l|^{3-\epsilon}}\big|\lesssim
 	|\eta|^{-2+\epsilon}+\sum\limits_{\substack{(l_1,l_2)\in \mathbb{Z}^2_{+}\\ 
|l|_\infty \leq N}}|\eta|\sup\limits_{\xi\in [l-\eta,l+\eta]}
 	\frac{1}{|\xi|^{3-\epsilon}},
\end{split}
\eeno
where $[l-\eta,l+\eta]= [l_1-\eta_1]\times  [l_2-\eta_2]    $.\\
Since $\eta\in [-\frac12,\frac12]^2, l\in \Z_{+}^2$, $l-\eta$ is uniformly away from the origin, we obtain
$\sup\limits_{\xi\in [l-\eta,l+\eta]}
 	\frac{1}{|\xi|^{3-\epsilon}}\lesssim \frac{1}{|l|^{3-\epsilon}}.
$  
Thus, we find 
$$\sup\limits_N\sum\limits_{\substack{(l_1,l_2)\in \mathbb{Z}^2_{+}\\ 
|l|_\infty \leq N}}|\eta|\sup\limits_{\xi\in [l-\eta,l+\eta]}
 	\frac{1}{|\xi|^{3-\epsilon}}\lesssim |\eta|\lesssim |\eta|^{-2+\epsilon}.
 	$$
As a result, 
$\big|\sum\limits_{{|l|_\infty \leq N}}
 	\frac{(\eta+l)^{\perp}}{|\eta+l|^{3-\epsilon}}\big|
$
has a uniform bound $C|\eta|^{-2+\epsilon} $, where $C$ is some constant independent of $N$.
Moreover, by the same argument we obtain that $\{\sum\limits_{{|l|_\infty \leq N}}
 	\frac{(\eta+l)^{\perp}}{|\eta+l|^{3-\epsilon}}\}_{N\geq 1}$
is a Cauchy sequence, hence the limit $\lim\limits_{N\rightarrow \infty}\sum\limits_{{|l|_\infty \leq N}}
 	\mathcal{F}^{-1}(\frac{i\xi^{\perp}}{|\xi|^{1+\epsilon}})(\eta+l)
 	$
 exists and 
\eqref{summationlconvergence} holds, which finishes the proof of 1.
\

\

\

\noindent
 2  follows from the fact that  each derivative of  
$\sum\limits_{\substack{(l_1,l_2)\in \mathbb{Z}^2_{+}\\ |l|_\infty \leq N}}
 	\big[\frac{(\eta+l)^{\perp}}{|\eta+l|^{3-\epsilon}}-\frac{(l-\eta)^{\perp}}{|l-\eta|^{3-\epsilon}}\big]$
converges uniformly  with respect to $\eta\in[-\frac12,\frac12]^2$,
which can be easily obtained by the same argument of 1.
\

\

\noindent For 3, when we view $\sum\limits_{l\in\Z^2}\mathcal{F}^{-1}(\frac{i\xi^{\perp}}{|\xi|^{1+\epsilon}})(\eta+l)$ as a function of $\eta$ on $\T_1^2$, the $k$th Fourier coefficient is
\beno
\begin{split}
\int_{[0,1]\times [0,1]}\sum\limits_{l\in\Z^2}\mathcal{F}^{-1}(\frac{i\xi^{\perp}}{|\xi|^{1+\epsilon}})(x+l)e^{-2\pi ik\cdot x}
          dx
         &= \sum\limits_{l\in\Z^2}\int_{[l_1,l_1+1]\times [l_2,l_2+1]}
          \mathcal{F}^{-1}(\frac{i\xi^{\perp}}{|\xi|^{1+\epsilon}})(y)e^{-2\pi ik\cdot (y-l)}dy\\
        &= \sum\limits_{l\in\Z^2}\int_{[l_1,l_1+1]\times [l_2,l_2+1]}
          \mathcal{F}^{-1}(\frac{i\xi^{\perp}}{|\xi|^{1+\epsilon}})(y)e^{- 2\pi ik\cdot y}dy\\
          &=\int_{\R^2}\mathcal{F}^{-1}(\frac{i\xi^{\perp}}{|\xi|^{1+\epsilon}})(y)e^{-2\pi ik\cdot y}dy\\
          &=\frac{1}{4\pi^2}\int_{\R^2} \mathcal{F}^{-1}(\frac{i\xi^{\perp}}{|\xi|^{1+\epsilon}})(\frac{y}{2\pi})e^{- ik\cdot y}dy\\
          &=\int_{\R^2}\mathcal{F}^{-1}(\frac{2\pi i\xi^{\perp}}{|2\pi\xi|^{1+\epsilon}})(y)e^{- ik\cdot y}dy\\
          &=\bigl(\frac{1}{2\pi}    \bigr)^{\epsilon}\frac{ik^{\perp}}{|k|^{1+\epsilon}},
\end{split}
\eeno
where the last equality is due to $\mathcal{F}\mathcal{F}^{-1}=id$.\\
The proof of 3 is finished.

 \end{proof}

 \noindent
  Thus to  conclude,  
 combining the case of $\R^2$, 
 we have proved the following lemma:
 \begin{lemma}\label{uniformboundofkernel}
{\sl Let 
 $$ K_{\epsilon}= \mathcal{\F}^{-1}\bigl(i \xi^{\perp} |\xi| ^{-(1+\epsilon)}\bigr) ,  $$
                  $$ K_{\epsilon}^M=\sum\limits_{k\in\mathbb{Z}^{2}\backslash \{(0,0)\} }
  \bigl(\frac{M}{2\pi}    \bigr)^{\epsilon}\frac{ik^{\perp}}{|k|^{1+\epsilon}}\frac{1}{M^2} e^{2\pi ik\cdot \cdot /M}$$
  be the kernel corresponding to   the operator 
 $ \nabla^{\perp} (-\Delta)^{-(1+\epsilon)/2} $
 on $\R^2$ and $\T_M^2$,  respectively.
  Then  there exists a common constant $C_{\epsilon}$ which does not depend on $M$,  such that 
  
 $$\big|K_{\epsilon}(x)\big| \leq \frac{C_{\epsilon}}{|x|^{2-\epsilon}},$$
 and
 $$\big|K_{\epsilon}^M(x)\big|\leq \frac{C_{\epsilon}}{|x|^{2-\epsilon}},$$
 for any $x\in \R^2$,  $x\in \T_M^2$,  respectively.

 }

 \end{lemma}
  \qed
  
  \noindent
   Moreover,  note that if we fix some $x\neq 0$ and let $M$ goes to infinity,  the sum  
  $$\sum\limits_{k\in\mathbb{Z}^{2}\backslash \{(0,0)\} }
  \bigl(\frac{M}{2\pi}    \bigr)^{\epsilon}\frac{ik^{\perp}}{|k|^{1+\epsilon}}\frac{1}{M^2} e^{2\pi ik\cdot x /M}
  =\frac{1}{4\pi^2}\sum\limits_{k\in\mathbb{Z}^{2}\backslash \{(0,0)\} }
  \frac{i\frac{2\pi}{M}k^{\perp}}{\bigl(\frac{2\pi}{M}|k|\bigr)^{1+\epsilon}}\bigl(\frac{2\pi}{M} \bigr)^2e^{2\pi ik\cdot x /M}
  $$
  converges to the integration
$$  \frac{1}{4\pi^2}\int_{\R^2}i\xi^{\perp} |\xi| ^{-(1+\epsilon)}e^{ix\cdot \xi }d\xi,
  $$
  which is exactly the Fourier inverse transform of $i\xi^{\perp} |\xi| ^{-(1+\epsilon)}$.\\
  \
  
  \noindent
  In other word, 
  we have the following lemma
  \begin{lemma}
{\sl   For any $x\in \R^2\backslash \{0\}$,
 $K_{\epsilon}^M(x)$  converges pointwisely to $K_{\epsilon}$ as $M$ goes to infinity.
  
  }
  \end{lemma}
\qed

  \subsection{Weak formulation of mSQG equations}\label{section23}
  In this section we do some (at least formally) transformation to transform the
  equation to a weak form.
  A similar transformation can be found in 
  \cite{weakvorticityFlandoli, FS}.  We  put it here for
  completeness.\\
    Recall the mSQG equation on both $\T_M^2$ and $\R^2$:
  $$ \partial_t \omega+ u\cdot\nabla \omega = 0. $$
  Let $\phi$ be a test function, i.e.    $\phi\in C^{\infty}(\T_M^2)$ in the case of 
  $\T_M^2$ and 
  $\phi\in C_c^{\infty}(\R^2)$ in the case of 
  $\R^2$.
  Then we obtain 
  \beq\label{mSQGtestform}
   \langle \omega_{t},\phi\rangle =\left\langle \omega_{0}%
,\phi\right\rangle +\int_{0}^{t}\left\langle u(s)\cdot \nabla \omega
_{s},\phi  \right\rangle ds.
\eeq
  Note that 
  $u=K_\epsilon\ast \omega$ (on $\R^2$) or $u=K_\epsilon^M\ast \omega$
  (on $\T_M^2$),
  and both $K_\epsilon$ and $K_\epsilon^M$ are anti-symmetric. 
  Therefore,    we can transform \eqref{mSQGtestform} to 
  \beq\label{weakmsqgr2}
\left\langle \omega_{t},\phi\right\rangle =\left\langle \omega_{0}%
,\phi\right\rangle +\int_{0}^{t}\left\langle \omega_{s}\otimes\omega
_{s},H_{\phi,\epsilon}\right\rangle ds
\eeq
on $\R^2$,
and
\beq\label{weakmsqgtm2}
\left\langle \omega_{t},\phi\right\rangle =\left\langle \omega_{0}%
,\phi\right\rangle +\int_{0}^{t}\left\langle \omega_{s}\otimes\omega
_{s},H_{\phi,\epsilon}^M\right\rangle ds
\eeq
on $\T_M^2$,  
  where 
  \[
H_{\phi,\epsilon}\left(  x,y\right)  :=\frac{1}{2}K_{\epsilon}\left(  x-y\right)  \left(
\nabla\phi\left(  x\right)  -\nabla\phi\left(  y\right)  \right),
\]
  and 
 \[
H_{\phi,\epsilon}^M\left(  x,y\right)  :=\frac{1}{2}K_{\epsilon}^M\left(  x-y\right)  \left(
\nabla\phi\left(  x\right)  -\nabla\phi\left(  y\right)  \right).
\]

  \subsection{Introduction of space white noise}

\subsubsection{Space white noise  on
$\mathbb{T}^2$}

First,  we recall the definition and the construction  of the space white noise
distribution on the torus $\mathbb{T}^2=\mathbb{T}_{2\pi}^2=(\mathbb{R}/2\pi \mathbb{Z})^2$ 
(see,  for example, 
\cite{weakvorticityFlandoli}).
The following definition and construction mainly come from \cite{weakvorticityFlandoli}, which
we write  here for completeness.

A \textbf{space white noise (variable)} $\omega$ on
 $\mathbb{T}^2$ is  a Gaussian distributional valued random variable
mapping from some probability space $(\Xi, \mathcal{F},\P)$  to $ C^{\infty} (  \mathbb{T}^2)'$  such that 
\begin{itemize}
\item For any $\phi \in  C^{\infty} (  \mathbb{T}^2)$,  $\langle \omega,\phi\rangle$ is a real valued Gaussian random variable with zero mean.
\item For any $\phi,\psi\in  C^{\infty} (  \mathbb{T}^2)$,
 $$\mathbb{E}\langle \omega,\phi\rangle
\langle \omega,\psi\rangle=\langle\phi, \psi\rangle_{L^2(\T^2)}.$$

\end{itemize}
We call the distribution of a space white noise on $ C^{\infty} (  \mathbb{T}^2)'$ the \textbf{space white noise distribution  } (on $ C^{\infty} (  \mathbb{T}^2)'$).
Now we show the existence of the space white noise variable by  constructing it.\\
Define
\begin{equation*}\label{spacewhitenoise}
\omega=\sum\limits_{n\in \mathbb{Z}^2}G_{n}(\theta)\frac{1}{2\pi}e^{inx},
\end{equation*}
where  $\theta\in\Xi$,
$G_{n}=\overline{G_{-n}},$ 
and $G_{n}$,  $n\in \mathbb{Z}^2_{+}\cup\{0\}$ are independent random variables with standard (complex) Gaussian distributions.
Thus
we have for $m,n\in \mathbb{Z}^2_{+}$
$$\mathbb{E}[G_{n}G_{m}]=\delta_{mn}.  $$
Hence it is easy to verify that $\omega$ is a space white noise, the details of which
can be found in \cite{weakvorticityFlandoli}.
\begin{remark}
\
{\sl

\begin{itemize}
%\item The probability space $(\Xi, \mathcal{F},\P)$ needs to be large enough so that there are enough independent Gaussian random variables on it.
\item We know that $\omega \in H^{-1-\epsilon}$ $\P$-a.s. for any $\epsilon>0$, 
the proof of which can be found in  \cite[Section 2.1]{weakvorticityFlandoli}.
  Therefore,
the space white noise distribution is supported in
$H^{-1-}$.

\item By the definition of the space white noise on  $\mathbb{T}^2$,  any random variable with space white noise distribution on  $\mathbb{T}^2$ in some probability space could be expanded by the series $\omega=\sum\limits_{n\in \mathbb{Z}^2}G_{n}(\theta)\frac{1}{2\pi}e^{inx}$, where 
$G_{n}$,  $n\in \mathbb{Z}^2_{+}\cup\{0\}$ are independent random variables with standard Gaussian distributions in the same probability 
space.

\item  From now on we do not distinguish the notion of  a space white noise (variable) and 
the  space white noise distribution when no confusion occurs.

\end{itemize}}
\end{remark}

\subsubsection{Space white noise  on
$\mathbb{T}_M^2$}
Similarly, we  define  space white noise on $\mathbb{T}_M^2=(\mathbb{R}/M\mathbb{Z})^2$ in the same way.
A \textbf{space white noise (variable)} $\omega$ on
 $\mathbb{T}_M^2$ is  a Gaussian distributional valued random variable
mapping from some probability space $(\Xi, \mathcal{F},\P)$  to $ C^{\infty} (  \mathbb{T}_M^2)'$  such that 
\begin{itemize}
\item For any $\phi \in  C^{\infty} (  \mathbb{T}_M^2)$,  $\langle \omega,\phi\rangle$ is a real valued Gaussian random variable.
\item For any $\phi,\psi\in  C^{\infty} (  \mathbb{T}_M^2)$,
 $$\mathbb{E}\langle \omega,\phi\rangle
\langle \omega,\psi\rangle=\langle\phi, \psi\rangle_{L^2(\mathbb{T}_M^2)}.$$

\end{itemize}
We call the distribution of a space white noise on $ C^{\infty} (  \mathbb{T}_M^2)'$ the \textbf{space white noise distribution  } (on $ C^{\infty} (  \mathbb{T}_M^2)'$).

\noindent
\textbf{Fourier transform and Sobolev spaces on $\T_M^2$}:\\
set
\beq\label{basisM}
 \{e_{n}^M\}_{n\in \mathbb{Z}^2 }=\{\frac{1}{M}e^{2\pi in\cdot x/M}\}_{n\in \mathbb{Z}^2 } 
 \eeq
 as the orthonormal basis of $L^2(\mathbb{T}_M^2, \mathbb{C})$.\\
For $u\in C^{\infty}(\mathbb{T}_M^2)$, we consider the following Fourier expansion of $u$ on the torus:
 $$u(x)=\sum\limits_{k\in \mathbb{Z}^{2} }\hat{u}_{k}^Me_{k}^M\with
 \hat{u}_{k}^M=\overline{\hat{u}_{-k}^M},$$
 where $\hat{u}_{k}^M:=\frac{1}{M}\int_{\mathbb{T}_{M}^2}u(x)e^{-2\pi ik\cdot x/M}dx$ denotes the $k$th  Fourier coefficient of $u$ on $\mathbb{T}_{M}^2$. 
% It follows from Fourier-Plancherel equality that the series is convergent  in $L^{2}(\mathbb{T}^{2})$.

\noindent
 Define the Sobolev norm on $\mathbb{T}_{M}^2$ for $s\in\mathbb{R}$ :
 $$\| u \|_{H^{s}(\mathbb{T}_M^2)}^{2}:=\sum\limits_{k\in \mathbb{Z}^{2} }\Bigl(1+(\frac{2\pi| k |}{M} )^{2}\Bigr)^s| \hat{u}_{k}^M|^{2s}.$$
Define the space 
$H^s(\mathbb{T}_M^2)$ as the completion of $C^{\infty}(\mathbb{T}_M^2)$
under the norm $\|\cdot \|_{H^s(\mathbb{T}_M^2)}$.\\
Similar to the case of $\T^2$,  a space white noise variable has the following form:
for some probability space $(\Xi, \mathcal{F},\P)$, define
\begin{equation}\label{TM}
\omega^M=\sum\limits_{n\in \mathbb{Z}^2}G_{n}^M(\theta)e_{n}^M, \textbf{ } \theta\in \Xi ,
\end{equation}
where  
$e_{n}^M=\frac{1}{M}e^{2\pi in\cdot x/M},$
$G_{n}^M=\overline{G_{-n}^M},$
and
$G_{n}^M$,  $n\in \mathbb{Z}^2_{+}\cup\{0\}$ are independent random variables with standard (complex) Gaussian distributions.
Thus
we have 
$$\mathbb{E}[G_{n}^MG_{m}^M]=\delta_{mn}$$
for $m,n\in  \mathbb{Z}^2_{+}$.\\

\subsubsection{ Space white noise  on $\mathbb{R}^2$}
\
When it comes to the cases of the whole space   
 $\R^2$, 
first we recall the definition of a space white
 noise $\omega$ on
 $\mathbb{R}^2$ as  a Gaussian distributional valued random variable
mapping from some probability space $(\Xi, \mathcal{F},\P)$  to $ C_c^{\infty} (  \mathbb{R}^2)'$  such that 
\begin{itemize}
	\item For any $\phi \in  C_c^{\infty} (  \mathbb{R}^2)$,  $\langle \omega,\phi\rangle$ is a real valued Gaussian random variable.
\item For any $\phi,\psi\in  C_c^{\infty} (  \mathbb{R}^2)$,
 $$\mathbb{E}\langle \omega,\phi\rangle
\langle \omega,\psi\rangle=\langle\phi, \psi\rangle_{L^2(\R^2)}.$$

\end{itemize}
We will construct a space white noise on $\R^2$ by taking the limit of space white noise on the torus and letting the volume of the torus go to infinity.
We extend \eqref{TM} periodically to a distribution $\bar{\omega}^{M}$ on $\mathbb{R}^2$ (with the Fourier series we can expand it directly by viewing it as the series on $\R^2$).
That is,
\begin{equation*}
\bar{\omega}^M=\sum\limits_{n\in \mathbb{Z}^2}G_{n}^M(\theta)e_{n}^M
\text { in } \mathbb{R}^2.
\end{equation*}
\\
However,  $\bar{\omega}^M$ is not uniformly  bounded with respect to $M$  in the sense of $H^{-1-}(\mathbb{R}^2)$ norm but only uniformly bounded in some weighted Sobolev spaces.
We have the following lemma.

\begin{lemma}\label{whitenoiseconvergence}
{\sl  For any $\nu>0$ and  $\sigma>2$, the distribution of 
 $\{\bar{\omega}^{M}\}_{M\geq 1}$  is tight in the weighted
Sobolev space $H^{-1-\nu}(\rho_{\sigma})$
%, where $\rho=\frac{1}{\langle x\rangle^{\sigma}}.$
. Hence $\{\bar{\omega}^{M}\}_{M\geq 1}$  is tight in the metric space $H^{-1-}(\rho_{\sigma})$, where $\rho_{\sigma}(x)= \frac{1}{\langle x\rangle^{\sigma}}    $.
Moreover,  denote by $\mu_{M}$  the distribution of $\bar{\omega}^{M}$ in $H^{-1-}(\rho_{\sigma})$. Then  $\mu_{M}$ converges weakly to the space white noise distribution on $\mathbb{R}^2$ in $H^{-1-}(\rho_{\sigma})$ as $M\rightarrow \infty$.  }
\end{lemma} 
\proof 
By the definition of the weighted Sobolev norm,
\begin{equation*}
\begin{split}
&\mathbb{E}\|\bar{\omega}^M\|_{H^{-1-\nu}( \rho_{\sigma})}^2\\
&=\mathbb{E}\|\rho_{\sigma} \bar{\omega}^M\|_{H^{-1-\nu}(\mathbb{R}^2)}^2\\
&=\mathbb{E}\|\sum\limits_{n\in \mathbb{Z}^2}\rho_{\sigma} e_{n}^MG_{n}^M\|_{H^{-1-\nu}(\mathbb{R}^2)}^2\\
&=
\int_{\mathbb{R}^2}\mathbb{E}\big|\sum\limits_{n\in \mathbb{Z}^2}G_{n}^M\int_{\mathbb{R}^2}\frac{\frac{1}{M}e^{2\pi in\cdot x/M}}{\langle x\rangle^{\sigma}}
e^{-i\xi\cdot x}  dx\big|^2(1+|\xi|^2)^{-1-\nu}d\xi\\
&\lesssim \frac{1}{M^2}\int_{\mathbb{R}^2}\sum\limits_{n\in  \mathbb{Z}^2}\big( \int_{\mathbb{R}^2}\frac{e^{2\pi in\cdot x/M-i\xi\cdot x}}{\langle x\rangle^{\sigma}}
 dx\big)^2(1+|\xi|^2)^{-1-\nu}d\xi,
\end{split}
\end{equation*}
where the last inequality is due to the reason that 
$G_{n}^M=\overline{G_{-n}^M},$
and
$G_{n}^M$,  $n\in \mathbb{Z}^2_{+}\cup\{0\}$ are independent random variables with standard (complex) Gaussian distributions.\\
Note that 
\begin{equation}\label{Besselpotentialconvergence}
\begin{split}
\frac{1}{M^2}\sum\limits_{n\in  \mathbb{Z}^2}\big( \int_{\mathbb{R}^2}\frac{e^{2\pi in\cdot x/M-i\xi\cdot x}}{\langle x\rangle^{\sigma}}
 dx\big)^2
&=\frac{1}{M^{2}}\sum\limits_{n\in  \mathbb{Z}^2}\big[\mathcal{F}\big(\rho_{\sigma}\big)(\xi-\frac{2\pi n}{M})\big]^2\\
&=\frac{1}{4\pi^2}\sum\limits_{n\in  \mathbb{Z}^2}\big(\frac{2\pi}{M}\big)^2\big[\mathcal{F}\big(\rho_{\sigma}\big)(\xi-\frac{2\pi n}{M})\big]^2.
\end{split}
\end{equation}
Since $\sigma>2$, $\rho_{\sigma}=\frac{1}{(1+|x|^2)^{\frac{\sigma}{2}}}\in L^1   $, 
$\mathcal{F}\big(\rho_{\sigma}\big)$ is continuous and bounded.
Moreover, since  $\rho_{\sigma}$
is infinitely smooth with all the derivatives bounded and $L^1$-integrable, $\mathcal{F}\big(\rho_{\sigma}\big)(x)$ decays  faster than $(1+|x|)^{-N} $ for any $N>0$ when $x$ goes to infinity.
%Thus $\sum\limits_{n\in  \mathbb{Z}^2_{+}\cup\{0\}}\big[\mathcal{F}\big(\rho_{\sigma}\big)(\xi\cdot M-2\pi n)\big]^2$ is bounded by a constant independent of $M$ and $\xi$.
Therefore,  \eqref{Besselpotentialconvergence} is   bounded  by 
$\frac{C}{M^{2}}\sum\limits_{n\in  \mathbb{Z}^2}(1+ | \xi-\frac{2\pi n}{M} |)^{-4}\lesssim M^2,
$  which is independent of $\xi$.
And when $M$  goes to infinity, \eqref{Besselpotentialconvergence}
will converge to 
$$\frac{1}{4\pi^{2}}\int_{\R^2}\big[\mathcal{F}\big(\rho_{\sigma}\big)(x)\big]^2 \ dx,
$$
where the rate of convergence is obviously independent of $\xi$.
Hence $\eqref{Besselpotentialconvergence}$ is uniformly bounded for any $\xi\in \R^2$ and $M\geq 1$.
Therefore,
\begin{equation*}
\begin{split}
\mathbb{E}\|\bar{\omega}^M\|_{H^{-1-\nu}( \rho_{\sigma})}^2
&\lesssim \int_{\mathbb{R}^2}
(1+|\xi|^2)^{-1-\nu}d\xi
\\
&\lesssim 1.
\end{split}
\end{equation*}
Note that $H^{-1-\nu}(\rho_{\sigma})$ can be compactly embedded into
 $H^{-1-2\nu}(\rho_{\sigma'} )$, where $\rho_{\sigma'}(x):= \frac{1}{\langle x\rangle^{\sigma'}}    $ and $\sigma'>\sigma$.
Hence  we obtain the tightness of $\{\mu_M\}_{M\geq 1}$ in $H^{-1-}(\rho_{\sigma})$.

\noindent
Therefore,  by the Prokhorov's theorem and Skorokhod's  theorem, there exists a  subsequence $M_k$ such that $\mu_{M_k}$ converges weakly to some limit $\bar{\mu}$.
Moreover, there exists a sequence of random variables
$\bar{\omega}'^{M_k}$ on another probability space $(\Xi', \mathcal{F}',\P')$, which have distributions $\mu_{M_k}$, such that 
$\bar{\omega}'^{M_k}$ converges  to $\bar{\omega}$ in $H^{-1-}(\rho_{\sigma})$  $\P'-a.e. $, such that $\bar{\omega}$
has the distribution $\bar{\mu}$.
\noindent
We claim that $\bar{\mu}$ is the space white noise distribution on 
$\mathbb{R}^2$.
Indeed,  
%since $\bar{w}^{M_k}$ converges weakly to $\bar{w}$, 
we know that for $\phi\in C_c^\infty(\mathbb{R}^2)$,  there exists some $k_0$ such that 
for  $k\geq k_0$,  $\phi$ is supported on the ball with the radius smaller than $\frac{M_k}{2}$, then we can view $\phi$  as a function $\phi_M$ on the torus  $\mathbb{T}_M^2$, thus for  $k\geq k_0$,  
$$\langle \bar{\omega}^{M_k},  \phi  \rangle= \langle \omega^{M_k},  \phi_{M_k}  \rangle
$$ is centred Gaussian, therefore,  $\langle \bar{\omega},  \phi  \rangle$ is centred Gaussian.
Moreover,  similarly,  from the argument of the explanation of Definition 4 of \cite{10.1214/16-AOP1116},
if we fix $\phi, \psi\in C_{c}^{\infty}$,
when $k$ is large enough such that $\phi$ and $\psi$ are supported in the ball with the radius smaller than $\frac{M_k}{2}$,
we have 
$$\mathbb{E}\langle \bar{\omega}^{M_k},  \phi  \rangle\langle \bar{\omega}^{M_k},  \psi  \rangle=\langle\phi, \psi\rangle_{L^2(\R^2)},$$
thus we have
$$\mathbb{E}\langle \bar{\omega},  \phi  \rangle\langle \bar{\omega},  \psi  \rangle=\langle\phi, \psi\rangle_{L^2(\R^2)},$$
which finishes the proof of our claim.

\qed

\section{Main result}\label{sec3}

  \subsection{Definition of the nonlinear term and the white noise solutions of mSQG on $\R^2$}
After the preparations, we will introduce our main result.
First we introduce the definition of the white noise (stationary) solution of   the weak formulation form of \eqref{2DmSQG}.

\begin{definition}\label{whitenoisesolution}
{\sl 
Fix any $T>0$. We say that a measurable map $\omega_{\cdot}:\Xi
\times\left[  0,T\right]  \rightarrow C_c^{\infty}\left(  \mathbb{R}^{2}\right)
^{\prime}$ (where $(\Xi,\mathcal{F},\P)$ is some probability space)
with trajectories of class $C (  \left[  0,T\right]
;  (C_c^{2} ) ') $ (see Definition \ref{weakstardual} for the definition of the topological space $(C_c^{2} ) '$)   is a white noise  (weak) solution of 
\eqref{2DmSQG}, if it satisfies the following: 
\begin{enumerate}
\item For fixed $t$, $\omega_t$ is a space white noise  on $\mathbb{R}^2$.
\item 
For any $\phi\in C_c^{\infty}(\mathbb{R}^2)$,
 
$$ \langle \omega_{t},\phi\rangle =\left\langle \omega_{0}%
,\phi\right\rangle +\int_{0}^{t}\left\langle \omega_{s}\otimes\omega
_{s},H_{\phi,\epsilon}\right\rangle ds $$
holds      $\P$-a.e.,
where \[
H_{\phi,\epsilon}\left(  x,y\right)  :=\frac{1}{2}K_{\epsilon}\left(  x-y\right)  \left(
\nabla\phi\left(  x\right)  -\nabla\phi\left(  y\right)  \right),
\]
and we will introduce the definition of the term $\left\langle \omega_{s}\otimes\omega
_{s},H_{\phi,\epsilon}\right\rangle$
later in Theorem  \ref{constrcutomegaomegaHphi}.

\end{enumerate}
}
\end{definition}
\noindent
Similarly to  the paper \cite{weakvorticityFlandoli}, here we also need to define 
the nonlinear term by constructing an approximating sequence.\\
 
Since for $\sigma>2$, the space white noise  $\bar{\omega}$ is $\P$-a.e. in the weighted Sobolev space 
$H^{-1-}(\rho_{\sigma})$ (see Lemma \ref{whitenoiseconvergence}),
$\bar{\omega}    \otimes    \bar{\omega}$ is in 
$H^{-2-}(\R^4,\rho_{\sigma}\times \rho_{\sigma} )$ $\P$-a.e.
$ \big\<\bar{\omega}    \otimes    \bar{\omega}, f  \big\>$ 
is defined  
when
$f\in H^{2+}(\R^4;  \rho_{\sigma}^{-1}\times\rho_{\sigma}^{-1} ) $, where
$H^{2+}(\R^4; \rho_{\sigma}^{-1}\times\rho_{\sigma}^{-1} ) $
is the union of the spaces $H^{2+\nu}(\R^4,  \rho_{\sigma}^{-1}\times\rho_{\sigma}^{-1} ) $ for all $\nu>0$.
 In particular, it can be defined when
$ f\in C_c^{\infty}(\R^2\times \R^2)  $.
However,  $H_{\phi,\epsilon} $ does not belong to the space
$H^{2+}(\R^4,  \rho_{\sigma}^{-1}\times\rho_{\sigma}^{-1} ) $.  Thus similarly to  \cite{weakvorticityFlandoli},    we  need to define the nonlinear term by constructing approximating sequence.

First of all,  the following lemma gives for  smooth and compactly supported
function  $ \phi$,
$H_{\phi,\epsilon} \in L^2(\R^2\times \R^2)$.
\begin{lemma}\label{HL2}
{\sl For $\phi\in  C_c^{2}(\R^2) $,
$H_{\phi,\epsilon} \in L^2(\R^2\times \R^2)$.
}
\end{lemma}
\begin{proof}
We prove directly by calculation.
Since
$$
H_{\phi,\epsilon}\left(  x,y\right)  :=\frac{1}{2}K_{\epsilon}\left(  x-y\right)  \left(
\nabla\phi\left(  x\right)  -\nabla\phi\left(  y\right)  \right),
$$
and
$K_{\epsilon}\left(  x-y\right) \leq \frac{1}{|x-y|^{2-\epsilon}}
$,
$0<\epsilon<1$,
assuming that $\phi$ is supported in the ball of radius $R$,
we have
\begin{equation*}
\begin{split}
&\int_{\R^2}\int_{\R^2}|H_{\phi,\epsilon}(x,y)|^2dxdy\\
&\leq \int_{\R^2}\int_{\R^2} \frac{|
\nabla\phi\left(  x\right)  -\nabla\phi\left(  y\right)  |^2}{|x-y|^{4-2\epsilon}}dxdy\\
&\leq \int_{|x| \leq 2R}\int_{|y| \leq 2R} \frac{\|D^2\phi\|_{L^{\infty}}^2}{|x-y|^{2-2\epsilon}}dxdy
+2\int_{|x| \leq R}\int_{|y| \geq 2R} \frac{|
\nabla\phi\left(  x\right)  -0 |^2}{|x-y|^{4-2\epsilon}}dxdy\\
&\leq C(R) \|D^2\phi\|_{L^{\infty}}^2
+2 \|\nabla \phi\|_{L^{\infty}}^2\pi R^2 \int_{|y| \geq 2R}\frac{1}{(|y|-R)^{4-2\epsilon}} dy\\
&\leq C(R)( \|D^2\phi\|_{L^{\infty}}^2+ \|\nabla \phi\|_{L^{\infty}}^2  ),
\end{split}
\end{equation*}
where $C(R)$ is a constant which only depends on  $R$ and
the second inequality is due to the symmetric property of 
$H_{\phi,\epsilon}(x,y)$.
\end{proof}
\\
Since we only used the property $K_{\epsilon}\left(  x-y\right) \leq \frac{1}{|x-y|^{2-\epsilon}}
$,  by Lemma \ref{uniformboundofkernel} we immediately  have the following corollary: 
\begin{corollary}\label{TMHphibound}
{\sl   Let $\phi\in C_c^{2}(\R^2)  $ be a function  supported in 
$[-\frac{M_0}{2},\frac{M_0}{2}]^2$.
Then for any $M>M_0$,   $\phi$ can be viewed as a function 
in $C^{2}(\T_M^2)$.  
For any $M>M_0$, we have 
$$
H_{\phi,\epsilon}^M\left(  x,y\right)  :=\frac{1}{2}K_{\epsilon}^M\left(  x-y\right)  \left(
\nabla\phi\left(  x\right)  -\nabla\phi\left(  y\right)  \right)\in L^2(\T_M^2\times\T_M^2).$$

Moreover, there exists a constant $C_{\phi}$ which does not depend on $M$, such that 
$$\|H_{\phi,\epsilon}^M\left(  x,y\right)\|_{L^2(\T_M^2\times\T_M^2)}\leq C_{\phi}.$$
}

\end{corollary}

\noindent
Similar as Theorem 8 of \cite{weakvorticityFlandoli},  we will prove 
the following theorem which gives the approximating sequence.
\begin{theorem}\label{constrcutomegaomegaHphi}
{\sl 
Fix  $\phi\in C_c^{2}(\R^2\times \R^2)$.
Assume that $f_n\in C_c^{\infty}(\R^2\times \R^2)  $ are symmetric and
approximate $H_{\phi,\epsilon} $ in the following sense:
\begin{align*}
\lim_{n\rightarrow\infty}\int_{\R^2}\int_{\R^2}\left(  f_n-H_{\phi,\epsilon}\right)
^{2}\left(  x,y\right)  dxdy  & =0\\
\lim_{n\rightarrow\infty}\int_{\R^2} f_n \left(  x,x\right)  dx  & =0.
\end{align*}
Then the sequence of r.v.'s $\left\langle \bar{\omega} \otimes \bar{\omega} ,f_n
\right\rangle $ is a Cauchy sequence in mean square. We denote by
\[
\left\langle \bar{\omega} \otimes\bar{\omega} ,H_{\phi,\epsilon}\right\rangle
\]
its limit. Moreover, the limit is the same if $f_n$ is replaced by
$\widetilde{f}_n$ with the same properties and such that
$\lim\limits_{n\rightarrow\infty}\int_{\R^2}\int_{\R^2}(\widetilde{f}_n-f_n)^2\left(  x,y\right)  dxdy=0$.
}

\end{theorem}
\begin{proof}
Without loss of generality, we assume that  for each $n$,  $f_n$ is supported 
in $[-\frac{n}{4},\frac{n}{4}]^4$.

In the following we will use the relation between white noise on the torus and the whole space, the details of which can be found in 
\cite{10.1214/16-AOP1116}.\\ Since $f_n$ is supported 
in $[-\frac{n}{4},\frac{n}{4}]^4$, it could also be viewed as a smooth function 
on $\T_M^2\times \T_M^2$ when $M\geq n$.

By  the explanation of  Definition 4 in  \cite{10.1214/16-AOP1116},
we know
\beq\label{whitenoisespacetorus}
\langle \bar{\omega} \otimes \bar{\omega} ,f_n
\rangle =\langle \omega^n \otimes \omega^n ,f_n
\rangle= \langle \omega^m \otimes \omega^m ,f_n
\rangle,
\eeq
for $m\geq n$, where $f_n$ is  understood as a function on
$\R^2$,   $\T_n^2$ and $\T_m^2$ respectively.
\\
Therefore,  now what we need to prove is the following:

$\langle \omega^n\otimes  \omega^n,
f_n
\rangle$ converges in $L^2(\Xi)$ as $n$ goes to $+\infty$ .\\
To prove the convergence it suffices to prove it is a Cauchy sequence.\\
Since $\lim\limits_{n\rightarrow\infty}\int_{\R^2} f_n\left(  x,x\right)  dx=0$, it
is equivalent to show that $\left\langle \omega^n\otimes\omega^n,f_n\right\rangle -\int f_n\left(  x,x\right)  dx$ is a Cauchy
sequence in mean square. We have for $m\geq n$,%
\begin{equation}\label{Cauchysequencewhitenoise}
\begin{split}
& \mathbb{E}\left[  \left\vert \left\langle \omega^n\otimes\omega^n,f_n \right\rangle -\int_{\R^2} f_n\left(  x,x\right)  dx-\left\langle
\omega^m\otimes\omega^m,f_m \right\rangle +\int_{\R^2} f_m\left(
x,x\right)  dx\right\vert ^{2}\right] \\
&=\mathbb{E}\left[  \left\vert \left\langle \omega^n\otimes\omega^n,f_n \right\rangle -\int_{\R^2}  f_n\left(  x,x\right)  dx-\left\langle
\omega^m\otimes\omega^m,f_m \right\rangle +\int_{\R^2}  f_m\left(
x,x\right)  dx\right\vert ^{2}\right] \\
& =\mathbb{E}\left[  \left\vert \left\langle \omega^m\otimes\omega^m ,\left(
f_n-f_m \right)  \right\rangle -\int_{\R^2} \left(  f_n-f_m\right)  \left(  x,x\right)  dx\right\vert ^{2}\right],
\end{split}
\end{equation}
where the second equality is due to \eqref{whitenoisespacetorus}.\\
By (ii) and (iii) of the  Corollary 6 in  \cite{weakvorticityFlandoli},
(for completeness we attach the corollary  later in Corollary \ref{corollary WN})
 we know that \eqref{Cauchysequencewhitenoise}
equals
\[
2\int_{\T_m^2}\int_{\T_m^2} \left(  f_n-f_m\right)  ^{2}\left(  x,y\right)  dxdy
=2\int_{\R^2}\int_{\R^2} \left(  f_n-f_m\right)  ^{2}\left(  x,y\right)  dxdy,
\]
which implies the Cauchy property of $\langle \bar{\omega} \otimes \bar{\omega} ,f_n
\rangle $ in mean square.
Hence $
\left\langle \bar{\omega} \otimes\bar{\omega} ,H_{\phi,\epsilon}\right\rangle
$ is well defined. 

Moreover,  by a similar way we prove that if we replace $f_n$ by
$\tilde{f}_n$ with the same properties
and such that
$\lim\limits_{n\rightarrow\infty}\int_{\R^2}\int_{\R^2}(\widetilde{f}_n-f_n)^2\left(  x,y\right)  dxdy=0$
,  $\langle \bar{\omega} \otimes \bar{\omega} ,\tilde{f}_n
\rangle $ also converges in mean square to  $
\left\langle \bar{\omega} \otimes\bar{\omega} ,H_{\phi,\epsilon}\right\rangle
$.
\end{proof}

\begin{corollary} See \cite[Corollary 6]{weakvorticityFlandoli} 
\label{corollary WN}

i) If $\omega^M:\Xi\rightarrow C^{\infty}\left(
\mathbb{T}_M^{2}\right)  ^{\prime}$ is a white noise and $f\in H^{2+}\left(
\mathbb{T}_M^{2}\times\mathbb{T}_M^{2}\right)  $, then for every $p\geq1$ there is
a constant $C_{p,M}>0$ such that
\[
\mathbb{E}\left[  \left\vert \left\langle \omega^M\otimes\omega^M,f\right\rangle
\right\vert ^{p}\right]  \leq C_{p,M}\left\Vert f\right\Vert _{{L^{\infty}}}^{p}.
\]

ii)\ We have $\mathbb{E}\left[  \left\langle \omega^M\otimes\omega^M
,f\right\rangle \right]  =\int_{\mathbb{T}_M^{2}}f\left(  x,x\right)  dx$.

iii)\ If $f$ is symmetric, then%
\[
\mathbb{E}\left[  \left\vert \left\langle \omega^M\otimes\omega^M,f\right\rangle
-\mathbb{E}\left[  \left\langle \omega^M\otimes\omega^M,f\right\rangle \right]
\right\vert ^{2}\right]  =2\int_{\mathbb{T}_M^{2}}\int_{\mathbb{T}_M^{2}}f\left(
x,y\right)  ^{2}dxdy.
\]

\end{corollary}

We now give an example of  the approximating sequence $\{f_n\}_{n\geq 1}$.
\

\

\noindent
\textbf{Constructuion of the approximating sequence $\{f_n\}_{n\geq 1}$}

We have proved in Lemma \ref{HL2} that for fixed $\phi\in C_c^{2}(\R^2) $,
$H_{\phi,\epsilon} \in L^2(\R^2\times \R^2)$,
thus there exists a sequence of  function $g_n\in C_c^{\infty}(\R^2\times \R^2)$ which converge to 
$H_{\phi,\epsilon}$ in $L^2(\R^2\times \R^2)$. 
We can also assume that $g_n$ is symmetric (otherwise,  
we let $\tilde{g}_n=\frac12\bigl(g_n(x,y)+g_n(y,x)\bigr)$.)\\
Without loss of generality we assume that  for each $n$,  $g_n$ is supported 
in $[-\frac{n}{4},\frac{n}{4}]^4$.
Let $f_n=r_ng_n$, where $r_n(x,y)$ is defined as follows:
$$ r_n \begin{cases}
  =1, &| x- y|\geq \frac{1}{n^6};\\
  =0, & | x- y|\leq \frac{1}{2n^6};\\
  \in [0,1] \text{  such that $r_n$ smooth},  & \frac{1}{2n^6}\leq | x- y|\leq \frac{1}{n^6}.
  \end{cases}$$
Thus $f_n$ is also smooth and supported in $[-\frac{n}{4},\frac{n}{4}]^4$.
Moreover,
\begin{equation*}
\begin{split}
\|f_n-H_{\phi,\epsilon}\|_{L^2}&\leq \|g_n-H_{\phi,\epsilon}\|_{L^2}
+\|(f_n-g_n)1_{| x- y|\leq \frac{1}{n^6}}\|_{L^2}\\
&\leq  \|g_n-H_{\phi,\epsilon}\|_{L^2}+\|g_n1_{| x- y|\leq \frac{1}{n^6}}\|_{L^2}\\
&\leq 2\|g_n-H_{\phi,\epsilon}\|_{L^2}+\|H_{\phi,\epsilon}1_{(x,y)\in [-\frac{n}{4},\frac{n}{4}]^4,| x- y|\leq \frac{1}{n^6}}\|_{L^2}\\
&\rightarrow 0,
\end{split}
\end{equation*}
where the last line is due to the fact that $g_n$ converges to $H_{\phi,\epsilon} $ in $L^2$ and the Lebesgue measure of the set
$\{(x,y)\in [-\frac{n}{4},\frac{n}{4}]^4;   | x- y|\leq \frac{1}{n^6}       \}
$ goes to $0$.

\begin{remark}\label{fnapprox}
\

{\sl
\begin{enumerate}
\item
Obviously, all the $f_n$ and $g_n$ that we defined above rely on 
$\phi$ and $\epsilon$, but for simplicity of the notation, we skip them
in our notation.
\item Note that the rate of convergence of $\left\langle \bar{\omega} \otimes \bar{\omega} ,f_n
\right\rangle $ to $\left\langle \bar{\omega} \otimes\bar{\omega} ,H_{\phi,\epsilon}\right\rangle$
in $L^2 (\Xi)  $
 only depends on the rate of the convergence of
$f_n$ to 
$H_{\phi,\epsilon}$
in $L^2(\R^2\times \R^2)$,
 but does not depend on $\bar{\omega}$ as long as it is 
a space white noise.
\item From our construction, we can require that 
$f_n(x,x)=0$ for any $n$.

\end{enumerate}}

\end{remark}

After we define the nonlinear term,  we manage to define the white noise (weak) solution of \eqref{2DmSQG} on $\R^2$.

\subsection{Main theorem of the paper}
Now we introduce the main result of the paper.

\begin{theorem}\label{maintheorem}
	{\sl There exists a white noise stationary (weak) solution of
	\eqref{2DmSQG} according to  Definition \ref{whitenoisesolution}.}
	\end{theorem}			
	In other words, we prove a similar result of  \cite{FS} by letting the volume of the torus go to infinity,  which is in the next section.

\section{Proof of the  Theorem \ref{maintheorem}}\label{sec4}
In this section we prove the main result (Theorem \ref{maintheorem}) of this paper.
First we recall the similar result on the torus.
Recalling  the  Theorem 1 of  \cite{FS} (also Theorem 1.1 of   \cite{LUO2021236} by letting $\theta=0$), the following theorem  was 
proved 
\begin{theorem}[Existence] \label{torusthm-existence}
{\sl Let $\epsilon\in (0,1)$. There exist a  probability space $(\Omega, \mathcal F,\P)$ and a stationary process $\xi:\Omega\to C\big( [0,T]; H^{-1-}\big)$ such that, for all $t\in [0,T]$, $\xi_t$ is a white noise on $\T^2$; and for all $\phi\in C^\infty(\T^2)$, $\P$-a.s. for all  $t\in [0,T]$, one has
  \beq\label{T2weakformulation}
  \begin{split}
  \big\<\xi_t,\phi \big\> =&\ \big\<\xi_0,\phi \big\> + \int_0^t \big\<\xi_s\otimes \xi_s, H_{\phi,\epsilon} \big\>\,d s .
  \end{split}
 \eeq }
\end{theorem}
Note that for $\epsilon=1$ it is the Euler equation,  the result is also proved in 
\cite{weakvorticityFlandoli}.
\noindent
\begin{remark}
{\sl
\begin{enumerate}
\
\item In both  \cite{weakvorticityFlandoli} and \cite{FS},  it is obvious that the zero set depends on $\phi$.  Indeed since the non-linear term is defined in the mean square sense, for any $\phi$ we can change the value of 
$\big\<\xi_s\otimes \xi_s, H_{\phi,\epsilon} \big\>$ in any  zero set $N_{\phi}$.
\item In \cite{weakvorticityFlandoli} and \cite{FS},  the proof of both cases
on the torus only use the boundedness of second derivatives of test function
$\phi$.  In other words,  the above theorem also holds when
$\phi\in C^2(\T^2)$.

\item In the proof of \cite{weakvorticityFlandoli} and \cite{FS}, the proof does not depend on the radius of the torus. Therefore, the same results also work for $\T_M^2$ for any $M>0$.
\end{enumerate}
}
\end{remark}

\noindent
Let $\omega_t^M$ be the solution in the above theorem on the torus $\T_M^2$ on some  probability space $(\Omega^M, \mathcal F^M,\P^M)$. 
First we fix
$\phi\in C_c^{2}$. Assume that $\phi$ is supported in  $[-\frac{A}{4}, \frac{A}{4}]^2$.
Thus for $M>A$,  $\phi$ could also be viewed as a function on the torus $\T_M^2$.
\noindent
By Theorem \ref{torusthm-existence} we have for $M>A$,  
$$\aligned
  \big\<\omega_t^M,\phi \big\> =&\ \big\<\omega_0^M,\phi \big\> + \int_0^t \big\<\omega_s^M\otimes \omega_s^M, H_{\phi,\epsilon}^M \big\>\,d s ,
  \endaligned $$
where
$H_{\phi,\epsilon}^M$ is defined in Section  \ref{section222} and Section \ref{section23}, and the nonlinear term is defined as in
\cite{weakvorticityFlandoli}.

Similar to before,  let $\bar{\omega}_t^M$ be the periodic 
extension of $\omega_t^M$  on $\R^2$.   Thus we have for any $t\geq 0$, 
$$\big\<\bar{\omega}_t^M,\phi \big\> =\big\<\omega_t^M,\phi_M \big\> ,
$$
where the left hand is defined on $\R^2$  and the right hand side is defined on the torus $\T_M^2$.
Thus 
\beq\label{Mequation}
  \big\<\bar{\omega}_t^M,\phi \big\> =\ \big\<\bar{\omega}_0^M,\phi \big\> + \int_0^t \big\<\omega_s^M\otimes \omega_s^M, H_{\phi_M,\epsilon}^M \big\>\,d s ,  
 \eeq
  where $\big\<\bar{\omega}_t^M,\phi \big\>$
  and $\big\<\bar{\omega}_0^M,\phi \big\>$ are duality products on $\R^2$
  but $\big\<\omega_s^M\otimes \omega_s^M, H_{\phi_M,\epsilon}^M \big\>$ is the duality product on 
  $\T_M^2$.
  
  \noindent
 Before we prove the tightness we need some more preparations.
 We begin with a lemma.

 \begin{lemma}
 {\sl The metric space $C_c^2$ with the $C^2$ H\"{o}lder
norm is separable. }
 \end{lemma}  
  \begin{proof}
  \
  
  \noindent
   \textbf{Step 1: to prove that $C_c^2$ can be approximated by $\mathcal{S}(\R^2)$}\\
  We fix a  family of smooth  functions which converge to the Dirac function,  for example,
  $$ \gamma_R(x)= \begin{cases}
  C_R\exp\{\frac{1}{R|x|^2-1} \}, &| x|^2\leq \frac{1}{R};\\
  0, & | x|^2\geq \frac{1}{R},
  \end{cases}$$
  where $C_R$ is a constant such that 
  
  $$\int_{\R^2} \gamma_R(x) \ dx=1.
  $$
   For any function  $f\in C_c^2$ and any index $|\alpha|\leq 2$, 
   $$D^\alpha(\gamma_R\ast f-f)=C_R\int_{\R^2}D^\alpha(f(x-y)-f(x))\gamma_R(y)\ dy.
   $$
   Since $f\in C_c^2$, $D^\alpha(f(x-y)-f(x))$ goes to $0$ uniformly as $y$ tends to $0$,
  hence $ \gamma_R\ast f$ converges to $f$ in $C^2$.
   It is obvious that $ \gamma_R\ast f$  is smooth and has compact support.  Therefore,   $ \gamma_R\ast f\in \mathcal{S}(\R^2)$. 
   \\
  \textbf{Step 2:  to find a Countable Dense Subset of $C_c^2$ } .\\
    Since $\mathcal{S}(\R^2)$ is separable,  let $\{f_i\}_{i \geq1 }$ be its  countable dense subset.      
  Since $C_c^{\infty}$ is dense in  $\mathcal{S}(\R^2)$,  for each 
  $f_i$, we can find a sequence $f_{ij}\in C_c^{\infty}    $,
  such that $f_{ij}$ converges to $f_i$ in $\mathcal{S}(\R^2)$
  (hence in $C^2$) as $j$ goes to infinity.  Therefore,  from the above arguments we know $\{f_{ij}\}_{i\geq 1, j\geq 1}$ is a dense subset of $C_c^2$ .\\
  Thus we have proved  $C_c^2$ is a separable metric space.
  \end{proof}

  \begin{definition}\label{weakstardual}
  {\sl We define the following function spaces:
  \begin{enumerate}
 \item 
 Define $(C_c^2)'$ to be the space which contains all the continuous linear  functional from $C_c^2$ to $\R$ with weak ${\ast}$ topology.

 \item 
  Define the time Sobolev  space $W^{1,2}([0,T];(C_c^2)') $ to be the space of all $u\in C([0,T];(C_c^2)') $ such that 
 $ u(\phi)  \in L^2([0,T];\R)$ and
 $\partial_t u(\phi)  \in L^2([0,T];\R)$ for any $\phi\in C_c^2$.
 The topology of $W^{1,2}([0,T];(C_c^2)') $ is defined to be the weakest topology on $W^{1,2}([0,T];(C_c^2)') $ such that
 for any $\psi\in L^2([0,T];C_c^2)$,
  the maps 
 $$u\mapsto \langle u ,\psi\rangle
 $$
 and 
  $$u  \mapsto \partial_t \langle u ,\psi\rangle 
 $$
are  continuous from $W^{1,2}([0,T];(C_c^2)') $ to $\R$. 
 \end{enumerate}

  }
  
  \end{definition}
  \begin{remark}
  {\sl
  \
  \begin{enumerate}

\item $C_c^2$ is not complete.  Denote by $C_0^2$ its closure with respect to  the  $C^2$ norm. Then by Banach--Steinhaus theorem the space  $(C_c^2)'$ is the same as the space
$(C_0^2)'$.  It is obvious that $C_0^2$ is also separable with the same 
countable dense subset of $C_c^2$.

\item   Since $C_c^2$ is separable, the closed  unit ball of $(C_c^2)'$ is compact
metric space by Banach-Alaoglu Theorem, hence also separable. 
Therefore,   $(C_c^2)'$ is also separable.

  \end{enumerate}}
  \end{remark}

We have the following tightness results.

\begin{lemma}
{\sl  Let $\{\mathcal{D}( \bar{\omega}_t^M)\}_{M=1}^{\infty}$ be the distribution of  $\bar{\omega}_t^M$ in $W^{1,2}([0,T];(C_c^2)')$. Then for any $T>0$,  $\{\mathcal{D}( \bar{\omega}_t^M)\}_{M=1}^{\infty}$ is tight in  $W^{1,2}([0,T];(C_c^2)') $. 
}

\end{lemma}

\begin{proof}
By definition of the topology of $(C_c^2)'$, it suffices to prove that for any $\phi\in C_c^2$ ,
\beq\label{msqgtight1}
\mathbb{E}|  \big\<\bar{\omega}_t^M,\phi \big\>|^2\leq C
\eeq
and 
\beq\label{msqgtight2}
\mathbb{E}|\partial_t \big\<\bar{\omega}_t^M,\phi \big\>|^2\leq C,
\eeq
where $C$ is a constant which  depends on $\phi$ but not $M$ and $t$.
  \eqref{msqgtight1}
 is immediately obtained by Lemma  \ref{whitenoiseconvergence}.
 
 \noindent
 To obtain \eqref{msqgtight2}, we note from \eqref{Mequation} that for any $t>0$,
 $$\partial_t \big\<\bar{\omega}_t^M,\phi \big\>=\big\<\omega_s^M\otimes \omega_s^M, H_{\phi_M,\epsilon}^M \big\>.$$
 Assume that $\phi$ is supported in  $[-\frac{A}{4}, \frac{A}{4}]^2$.
Then  for $M>A$,    $\phi_M=\phi$.\\
Hence we have for $M>A$,    
 \beq
 \mathbb{E}|\big\<\omega_s^M\otimes \omega_s^M, H_{\phi_M,\epsilon}^M \big\>|^2=\mathbb{E}|\big\<\omega_s^M\otimes \omega_s^M, H_{\phi,\epsilon}^M \big\>|^2.
  \eeq
 \noindent
By Corollary 6 of \cite{FS},
 we deduce
$$ \mathbb{E}|\big\<\omega_s^M\otimes \omega_s^M, H_{\phi,\epsilon}^M \big\>|^2=2\int_{\mathbb{T}_M^{2}}\int_{\mathbb{T}_M^{2}}H_{\phi,\epsilon}^M\left(
x,y\right)  ^{2}dxdy.
 $$ 
 Recall  that $H_{\phi,\epsilon}^M(x,y)=\frac12K_{\epsilon}^M(x-y)(\nabla \phi(x)-\nabla \phi(y) )  $, where
 $|K_\epsilon^M(x)| \leq C_\epsilon \frac{1}{|x|^{2-\epsilon}} $ and $C_\epsilon$ is a uniform constant does not depend on $M$.  Moreover, by Corollary  \ref{TMHphibound},
 $H_{\phi,\epsilon}^M$ is uniformly bounded with respect to $M$ in the sense of 
 $L^2(\T_M^2\times\T_M^2)$-norm.
 Hence $ \mathbb{E}|\big\<\omega_s^M\otimes \omega_s^M, H_{\phi,\epsilon}^M \big\>|^2$ is uniformly bounded.  Since there are only 
 finite positive integers which are smaller than $A$, we conclude \eqref{msqgtight2}.
\end{proof}

Now we apply the Skorokhod  Theorem  \ref{skoro}.
Note that the space $W^{1,2}([0,T];(C_c^2)')$ satisfies the requirement of Theorem  \ref{skoro},
since $W^{1,2}([0,T];(C_c^2)')$ is separated by the countable dense subset of 
$L^2([0,T];C_c^2)$.

By the statement of Theorem  \ref{skoro},   
one needs to show that 
the $\sigma$- algebra generated by 
the countable dense subset of 
$L^2([0,T];C_c^2)$ is exactly  the Borel $\sigma$- algebra
of $W^{1,2}([0,T];(C_c^2)')$.
By Theorem \ref{algebrasame}
it suffices to prove that $W^{1,2}([0,T];(C_c^2)')$ is a standard Borel space.
(See Appendix \ref{standardborelspace}  for the definition of the  standard Borel space).

\begin{lemma}
{\sl   $W^{1,2}([0,T];(C_c^2)')$ is a standard Borel space.
}
\end{lemma}
\begin{proof}
Let $X_1:=W^{1,2}([0,T];(C_c^2)')$ and 
$X_2:=W^{1,2}([0,T];L^2)$,
where $X_2$ consists of all the functions $v$ such that 
$v\in L^2([0,T]\times \R^2)$ and $\partial_t v\in L^2([0,T]\times \R^2)$
with the norm
$$\| v\|_{W^{1,2}([0,T];L^2)}:=\| v\|_{ L^2([0,T]\times \R^2)}
+\| \partial_t v\|_{ L^2([0,T]\times \R^2)}.
$$
It is obvious that $X_2$ is a Polish space and it is continuously embedded in $X_1$.
We need to prove 
$$\mathscr{B}(X_2)=\mathscr{B}(X_1 )\cap  X_2.
$$
Obviously $\mathscr{B}(X_1 )\cap  X_2    \subset  \mathscr{B}(X_2)           $.
\\
It suffices to show that  any open set of $X_2$ is in $\mathscr{B}(X_1 )\cap X_2$.\\
Note that   $\{B(x_m, r_n)\}_{m,n\geq 1}$ is  a countable topology basis of
$X_2$,
where $\{x_m\}_{m\geq 1}$ is a countable dense subset of $X_2$
and $\{r_n\}_{n\geq 1}$ is  the sequence of all the positive rational numbers.
Therefore,  we only need to prove
$$B(x_m, r_n)\in \mathscr{B}(X_1 )\cap X_2.
$$
Without loss of generality we only prove it for $x_m=0$.
\\
Note that
\beno
B(0, r_n)=&\{  x\in X_2;   \| x\|_{ L^2([0,T]\times \R^2)}
+\| \partial_t x\|_{ L^2([0,T]\times \R^2)}< r_n      \}\\ 
=&\bigcup\limits_{j\geq 1} \bigcap\limits_{k,l \geq 1} \{ x\in X_2;  |\langle x, \psi_k \rangle |+| \langle \partial_t x, \psi_l \rangle   | < r_n -\frac1j        \} \\
=& X_2\cap [ \bigcup\limits_{j\geq 1} \bigcap\limits_{k,l \geq 1} \{ x\in X_1;  |\langle x, \psi_k \rangle |+| \langle \partial_t x, \psi_l \rangle   | < r_n -\frac1j          \}]    ,
\eeno
where $\{\psi_k\}_{k\geq 1}$ is set to be a countable dense subset of 
the unit ball of $L^2([0,T]\times \R^2)$  such that  $\{\psi_k\}_{k\geq 1}$ is also a subset of 
$L^2([0,T];C_c^2)$.
Then  

$$\{x\in X_1; |\langle x, \psi_k \rangle |+| \langle \partial_t x, \psi_l \rangle   | < r_n -\frac1j        \}$$
 is an open set of $X_1$.
Hence $$B(0, r_n)\in \mathscr{B}(X_1 ) \cap X_2,
$$
which finishes our proof.
\end{proof}

Therefore,   there exists another probability space, 
which we still use the notation $(\Omega, \mathcal F,\P)$ for simplicity,
and a 
 sequence  of random variables   $\bar{\bar{\omega}}_t^{M_k}$ on  $(\Omega, \mathcal F,\P)$,   such that
 \begin{itemize}
 \item $\bar{\bar{\omega}}_t^{M_k}$ has the same distribution
to  $\bar{\omega}_t^{M_k}$ in $W^{1,2}([0,T];(C_c^2)') $;
(we also assume that $M_k$ is increasing to infinity and $M_k\geq k$) 
\item $\bar{\bar{\omega}}_t^{M_k}$ converge  $\P$-almost surely to some limit $\bar{\bar{\omega}}_t$ in $W^{1,2}([0,T];(C_c^2)') $.
 \end{itemize}

\noindent
Hence by the same argument of Lemma \ref{whitenoiseconvergence},
we obtain that for any fixed $t\in [0,T]$,  $\bar{\bar{\omega}}_t$ is a space white noise distribution on $\R^2$.

\noindent
By the definition of the solution on the torus, $\bar{\omega}_t^{M}$ has the following form
\begin{equation*}
\bar{\omega}_t^{M}=\sum\limits_{n\in \mathbb{Z}^2}\bar{G}_{n}^M(t,\theta_M)e_{n}^M
\text { on } \mathbb{R}^2,
\end{equation*}
where $\theta_M\in \Omega^M$,   $\bar{G}_{n}^M(\cdot,\theta_M) \in W^{1,2}([0,T];\R) $ and for each $t$,
$\bar{G}_{n}^M(t,\cdot)$,  $n\in \mathbb{Z}^2_{+}\cup\{0\}$ are independent random variables with standard Gaussian distributions
 on $(\Omega^M, \mathcal F^M,\P^M)$.\\
Note that for fixed $M$ ,   if $a_n(t)  \in W^{1,2}([0,T];\R) $ and\\ $\sum\limits_{n\in \mathbb{Z}^2}a_n(t) e_{n}^M\in W^{1,2}([0,T ];(C_c^2)') $ ,   the map
$$
\sum\limits_{n\in \mathbb{Z}^2}a_n(t) e_{n}^M
\mapsto \bigl(a_{n_1}(t),  a_{n_2}(t),..., a_{n_k}(t)\bigr)
$$
is continuous from  $W^{1,2}([0,T ];(C_c^2)') $ to
$\bigl(W^{1,2}([0,T];\R)\bigr) ^k $ for any $k$
and $n_1,n_2,...n_k \in \Z^2    $.

\noindent
Therefore,
$\bar{\bar{\omega}}_t^{M}$ also has the  form 
\begin{equation*}
\bar{\bar{\omega}}_t^{M}=\sum\limits_{n\in \mathbb{Z}^2}\bar{\bar{G}}_{n}^M(t,\theta)e_{n}^M
\text { on } \mathbb{R}^2
\end{equation*}
on  $(\Omega, \mathcal F,\P) $,
where $\bigl(\bar{G}_{n_1}^M(t,\cdot),\bar{G}_{n_2}^M(t,\cdot),...\bar{G}_{n_k}^M(t,\cdot)\bigr)$ and 
$\bigl(\bar{\bar{G}}_{n_1}^M(t,\cdot), \bar{\bar{G}}_{n_2}^M(t,\cdot),...,  \bar{\bar{G}}_{n_k}^M(t,\cdot)\bigr)$
have the same joint distributions on $\bigl(W^{1,2}([0,T];\R)\bigr) ^k $.
Define
\begin{equation*}
\hat{\omega}_t^{M}=\sum\limits_{n\in \mathbb{Z}^2}\bar{\bar{G}}_{n}^M(t,\theta)e_{n}^M
\text { on } \mathbb{T}_M^2,
\end{equation*}
i.e.   $\bar{\bar{\omega}}_t^{M}$ is  an extension of  $\hat{\omega}_t^{M}$
on $\mathbb{R}^2$.
 Moreover,  $\hat{\omega}_t^{M}$ has the same distribution as 
$\omega_t^{M}$,  hence it also satisfies the equation
\eqref{T2weakformulation}.

\noindent
Thus it satisfies the same equation as \eqref{Mequation} for any $\phi\in C_c^{2}(\R^2)$ $\P$-a.s.:

\beq\label{barbaromegaM}
  \big\<\bar{\bar{\omega}}_t^{M_k},\phi \big\> =\ \big\<\bar{\bar{\omega}}_0^{M_k},\phi \big\> + \int_0^t \big\<\hat{\omega}_s^{M_k}\otimes \hat{\omega}_s^{M_k}, H_{\phi_{M_k},\epsilon}^{M_k} \big\>\,d s .
 \eeq

\noindent
Same as usual, $\phi$   could also be viewed as a function on $\T_{M_k}^2$ when we fix $\phi$ and let $M_k$  large enough.
It suffices to prove for any fixed $\phi\in C_c^{2}(\R^2)$,
we have $\P$-a.s.
\beq\label{mainconvergence}
\lim\limits_{k \rightarrow \infty}\int_0^t \big\<\hat{\omega}_s^{M_k}\otimes \hat{\omega}_s^{M_k}, H_{\phi,\epsilon}^{M_k} \big\>\,d s 
=\int_0^t \big\<\bar{\bar{\omega}}_s\otimes \bar{\bar{\omega}}_s, H_{\phi,\epsilon}\big\>\,d s ,
\eeq
where on the left  hand side,  $ \big\<\omega_s^{M_k}\otimes \omega_s^{M_k}, H_{\phi,\epsilon}^{M_k} \big\>\ $ 
is the  duality product on the torus and on the right hand side,  
$\big\<\bar{\bar{\omega}}_s \otimes \bar{\bar{\omega}}_s, H_{\phi,\epsilon} \big\>\ $ is the  duality product on $\R^2$.
\\
\textbf{Proof of \eqref{mainconvergence}}\\
\textbf{Step 1}\\
Fix $\eta>0$.
Recall from Theorem \ref{constrcutomegaomegaHphi},  for a 
space white noise distribution $\bar{\omega}$ on $\R^2$ in some probability space,  we define 
$
\left\langle \bar{\omega} \otimes\bar{\omega} ,H_{\phi,\epsilon}\right\rangle
$ as the mean square limit of 
$\left\langle \bar{\omega} \otimes \bar{\omega} ,f_n
\right\rangle $, where 
$f_n\in C_c^{\infty}(\R^2\times \R^2)  $ are symmetric and
approximate $H_{\phi,\epsilon} $ $\P$-a.s. in the following sense:
\begin{align*}
\lim_{n\rightarrow\infty}\int_{\R^2}\int_{\R^2}\left(  f_n-H_{\phi,\epsilon}\right)
^{2}\left(  x,y\right)  dxdy  & =0\\
\lim_{n\rightarrow\infty}\int_{\R^2} f_n \left(  x,x\right)  dx  & =0.
\end{align*}
Moreover,  without loss of generality we assume that  for each $n$,  $f_n$ is supported 
in $[-\frac{n}{4},\frac{n}{4}]^4$.
 By   3 of Remark  \ref{fnapprox},  we can require 
 $f_n(x,x)=0$.
And by  2 of Remark  \ref{fnapprox}, we know the approximation 
is uniform with respect to the time $t$.
Hence we know 
$\big\<\bar{\bar{\omega}}_s \otimes \bar{\bar{\omega}}_s, H_{\phi,\epsilon} \big\>$ is 
the $L^2(\Omega; L^2([0,T]))=L^2([0,T]; L^2(\Omega))
$ limit of $\big\<\bar{\bar{\omega}}_s \otimes \bar{\bar{\omega}}_s, f_n \big\>$. 
Thus we can find an $n_0$, such that 
\beq\label{fn0}
\int_{\R^2}\int_{\R^2}\left(  f_{n_0}-H_{\phi,\epsilon}\right)^{2}\left(  x,y\right)  dxdy   <\frac{\eta}{T},
\eeq
thus
$$\mathbb{E}\int_{0}^T |\big\<\bar{\bar{\omega}}_s \otimes \bar{\bar{\omega}}_s, f_{n_0} \big\>-\big\<\bar{\bar{\omega}}_s \otimes \bar{\bar{\omega}}_s, H_{\phi,\epsilon} \big\>|^2 dt<\eta.
$$
\\
\textbf{Step 2}\\
 Fix $n_0$,  since 
$\bar{\bar{\omega}}_t^{M_k}$ converge  $ \P$-a.s.  to $\bar{\bar{\omega}}_t$ in $W^{1,2}([0,T];(C_c^2)') $, (hence in
$C([0,T];(C_c^2)') $)
$\big\<\bar{\bar{\omega}}_s^{M_k} \otimes \bar{\bar{\omega}}_s^{M_k}, f_{n_0} \big\>$ converges to 
$\big\<\bar{\bar{\omega}}_s \otimes \bar{\bar{\omega}}_s, f_{n_0} \big\>$ in $C([0,T];\R) $ $\P$-almost surely as $k$ goes to infinity.  Moreover,  since when $k\geq n_0$,  
$\big\<\bar{\bar{\omega}}_s^{M_k} \otimes \bar{\bar{\omega}}_s^{M_k}, f_{n_0} \big\>=\big\<\hat{\omega}_s^{M_k}\otimes \hat{\omega}_s^{M_k}, f_{n_0} \big\>$, 
where $f_{n_0}$   can be viewed as the product on the 
torus $\T_{M}^2$ when $k\geq n_0$,  just as we have shown during the proof of 
the Theorem \ref{constrcutomegaomegaHphi}.
By Corollary
\ref{corollary WN}
it is uniformly integrable.
Therefore,

$\big\<\hat{\omega}_s^{M_k}\otimes \hat{\omega}_s^{M_k}, f_{n_0} \big\>$  converges to 
$\big\<\bar{\bar{\omega}}_s \otimes \bar{\bar{\omega}}_s, f_{n_0} \big\>$
as $k\rightarrow\infty$ in $L^2(\Omega; L^2([0,T]))$.
\\
\textbf{Step 3}\\
By step 1 and step 2, we know that there exists some $k_0\geq n_0$ such that when $k\geq k_0$,  
\beq\label{step3}
\mathbb{E}\int_{0}^T | \big\<\hat{\omega}_s^{M_k}\otimes \hat{\omega}_s^{M_k}, f_{n_0} \big\> -\big\<\bar{\bar{\omega}}_s \otimes \bar{\bar{\omega}}_s, H_{\phi,\epsilon} \big\>|^2ds < 2 \eta
\eeq
\textbf{Step 4}\\
Just as we have mentioned in step 2, when $k\geq k_0\geq n_0$,
$\big\<\bar{\bar{\omega}}_s^{M_k}  \otimes \bar{\bar{\omega}}_s^{M_k} , f_{n_0} \big\>$ is the duality product 
on the 
torus $\T_{M_k}^2$.
\\
By ii) iii) of Corollary \ref{corollary WN} and the definition of  $\big\<\hat{\omega}_s^{M_k}\otimes \hat{\omega}_s^{M_k}, H_{\phi,\epsilon}^{M_k} \big\>$,
$$\mathbb{E}\int_{0}^T | \big\<\hat{\omega}_s^{M_k}\otimes \hat{\omega}_s^{M_k}, f_{n_0} \big\> -\big\<\hat{\omega}_s^{M_k}\otimes \hat{\omega}_s^{M_k}, H_{\phi,\epsilon}^{M_k} \big\>|^2ds \leq T\int_{\T_{M_k}^2}\int_{\T_{M_k}^2}\left(  f_{n_0}-H_{\phi,\epsilon}^{M_k}\right)^{2}\left(  x,y\right)  dxdy .
$$
If we view  $H_{\phi,\epsilon}^M(x,y)$ as  measurable functions on $\R^2$ which are $0$ valued outside 
$[-\frac{M}{2},\frac{M}{2}]^4$,
we can view $\int_{\T_{M_k}^2}\int_{\T_{M_k}^2}\left(  f_{n_0}-H_{\phi,\epsilon}^{M_k}\right)^{2}\left(  x,y\right)  dxdy$ as
$\int_{\R^2}\int_{\R^2}\left(  f_{n_0}-H_{\phi,\epsilon}^{M_k}\right)^{2}\left(  x,y\right)  dxdy$.\\
Since for any $x$, $K_\epsilon^M(x)$ goes to $K_\epsilon(x)$  as $M$ goes to infinity,   $H_{\phi,\epsilon}^M$ converges pointwisely to $H_{\phi,\epsilon}$.
Moreover,  $\{H_{\phi,\epsilon}^M\}_{M>0}$,  are all dominated by the $L^2(\R^2\times \R^2)$ integrable function $\frac{C\left(
\nabla\phi\left(  x\right)  -\nabla\phi\left(  y\right)  \right)}{|x-y|^{2-\epsilon}}   $
for some constant $C$  not depending on $M$, thus 
the convergence of $H_{\phi,\epsilon}^M$ to 
$H_{\phi,\epsilon}$ also holds in $L^2(\R^2\times \R^2)$.
Thus combining with \eqref{fn0},
we can find some $k_1\geq k_0$, such that when $k\geq k_1$,
$\int_{\R^2}\int_{\R^2}\left(  f_{n_0}-H_{\phi,\epsilon}^{M_k}\right)^{2}\left(  x,y\right)  dxdy<\frac{2\eta}{T}$, hence for any $k\geq k_1$,
\beq\label{step4}
\mathbb{E}\int_{0}^T | \big\<\hat{\omega}_s^{M_k}\otimes \hat{\omega}_s^{M_k}, f_{n_0} \big\> -\big\<\hat{\omega}_s^{M_k}\otimes \hat{\omega}_s^{M_k}, H_{\phi,\epsilon}^{M_k} \big\>|^2ds < 2\eta.
\eeq
Hence by \eqref{step3} and \eqref{step4}, we obtain that   \eqref{mainconvergence} holds in $L^2(\Omega)$.

Since for any $0\leq t\leq T$,  $\big\<\bar{\bar{\omega}}_t^{M_k},\phi \big\>$ converges to 
$\big\<\bar{\bar{\omega}}_t,\phi \big\>$ $\P$-a.s., the convergence of
\eqref{mainconvergence} also holds $\P$-a.s.

\bigbreak \noindent {\bf Acknowledgments.}  
S. Liang is grateful for the financial support from Deutsche Forschungsgemeinschaft(DFG) through the program IRTG 2235.
The author thanks  Prof. Dr. Rongchan Zhu 
for helpful discussion.

\newpage
\appendix

\section{Skorokhod's Representation Theorem}

We show the following Jakubowski's  version of the Skorokhod Theorem in the form given by Brze\'{z}niak and Ondrej\'{a}t \cite{brze?niak2013} Theorem A.1 and it was proved by A. Jakubowski in \cite{Jakubowski1998Short}.
\begin{theorem}\label{skoro}
{\sl Let $\mathcal{Y}$ be a topological space such that there exists a sequence ${f_{m}}$ of continuous functions $ f_{m}:\mathcal{Y}\rightarrow \mathbb{R}$ that separates points of $\mathcal{Y}$.    Let us denote by $\mathcal{S}$ the $\sigma$-algebra generated by the maps ${f_{m}}$. Then
\begin{itemize}
\item[(j1)] every compact subset of  $\mathcal{Y}$ is metrizable;

\item[(j2)] if $(\mu_{m})$ is tight sequence of probability measures on $(\mathcal{Y},\mathcal{S})$, then there exists a subsequence $(m_{k})$, a probability space $(\Omega , \mathcal{F},\mathbb{P})$  with $\mathcal{Y}$-valued Borel measurable variables $\xi_{k}$, $\xi$ such that $\mu_{m_{k}}$ is the law of $\xi_{k}$ and $\xi_{k}$ converges to $\xi$ almost surely on $\Omega$. Moreover, the law of $\xi$ is a Radon measure.
    \end{itemize}}
\end{theorem}

\section{Standard Borel Spaces   }
      First we introduce the following definitions  of countably generated  Borel Space and 
standard Borel space.
\begin{definition}[Countably generated Borel space, see \cite{PARTHASARATHY1967v} Chapter V Definition 2.1]\label{countablygenerated}
{\sl A Borel space $(X,\mathcal{B})$ is  said to be countably generated if there exists a denumerable class $\mathcal{D}\subset \mathcal{B}$ such that $\mathcal{D}$ generates $ \mathcal{B}$.
}

\end{definition}

\begin{definition}[Standard borel space, see \cite{PARTHASARATHY1967v} Chapter V Definition 2.2] \label{standardborelspace}
{\sl A countably generated Borel space $(X,\mathscr{B})$ is called standard if there exists a complete separable metric space Y such that the $\sigma$-algebras $\mathscr{B}$ and $\mathscr{B}(Y)$ are $\sigma$-isomorphic.
}

\end{definition}
Moreover, we will introduce the following theorem, which is  Theorem 2.4 of Chapter V of
 \cite{PARTHASARATHY1967v}.
 
 \begin{theorem}\label{thm24partha}
 {\sl Let $(X,\mathscr{B})$ be standard,  $(Y, \mathscr{C})$ countably generated and $\varphi$ a one-one map of $X$ into $Y$ which is measurable. Then $Y'=\varphi(X)\in \mathscr{C}$ and $\varphi$ is a Borel isomorphism between the Borel spaces $(X,\mathscr{B})$ and $(Y', \mathscr{C}_{Y'})$.
 }
 \end{theorem}

 By  Theorem 
 \ref{thm24partha}, we know the following theorem holds.

\begin{theorem}\label{algebrasame}
{\sl
 Let $(X,\mathscr{B})$ be any standard Borel space.  Assume that $\{f_n\}_{n\geq 1}$ is a sequence of $\mathscr{B}-$measurable   functions from    $X$ to $\R$ which separate the points of $X$.   Denote by $\sigma_0(X)$ the $\sigma$-algebra  generated by $\{f_n\}_{n\geq 1}$.  Then $\sigma_0(X)=\mathscr{B}$.}
 \end{theorem}      
    \begin{proof}
    Consider the identity map $id$: 
      $$(X, \mathscr{B}) \longrightarrow 
    (X,\sigma_0(X)).$$
      Since each $f_i$ is measurable,   it is obvious that $id$  is measurable.
    Hence by Theorem \ref{thm24partha} we know that $id$ is a Borel isomorphism, which finishes our proof.

\end{proof}

\cleardoublepage
\setcounter{secnumdepth}{-1}
%\addcontentsline{toc}{chapter}{Bibliography}
%\addcontentsline{toc}{section}{References}
\bibliographystyle{alpha}
%\printbibliography{Referenzen}
%\bibliography{Referenzen}

\begin{thebibliography}{999}

\bibitem{Bahouri2011Fourier}
  H. Bahouri, J.-Y. Chemin and R. Danchin,
{\it Fourier Analysis and Nonlinear Partial Differential Equations}, Grundlehren der mathematischen Wissenschaften, {\bf 343} Springer-Verlag Berlin Heidelberg, 2011.

\bibitem{brze?niak2013}
Z. Brze{\'{z}}niak and M. Ondrej{\'{a}}t,
\newblock Stochastic geometric wave equations with values in compact Riemannian
  homogeneous spaces.
\newblock {\em Ann. Probab.}, 41(3B):1938--1977, 05 2013.






\bibitem{bsv16}
 T.~Buckmaster, S.~Shkoller, and V.~Vicol,
 \newblock Nonuniqueness of weak
  solutions to the {SQG} equation
 \newblock {\em  Communications on Pure and Applied Mathematics},
 72:1809-1874,  2019.


\bibitem{arXiv210411048}
{ D.  Cao, G.  Qin,  W.  Zhan, C.  Zou
}, {On the global classical solutions for the generalized SQG equation }, arXiv preprint arXiv:2104.11048,  2021.






\bibitem{ChaConCorGanWu2012}
D.  Chae, P.  Constantin, D. C\'{o}rdoba, F.  Gancedo, and
  Jiahong Wu, 
  \newblock
  Generalized surface quasi-geostrophic equations with
  singular velocities, Comm. Pure Appl. Math. \textbf{65} no.~8,
  1037-1066, 2012. 


\bibitem {Chemin}J.-Y. Chemin, \textit{Perfect incompressible fluids}, volume
14 of Oxford Lecture Series in Mathematics and its Applications, The Clarendon
Press Oxford University Press, New York, 1998.


\bibitem{ConMajTab1994}
P.  Constantin, Andrew~J. Majda, and E. Tabak, \emph{Formation of strong
  fronts in the {$2$}-{D} quasigeostrophic thermal active scalar}, Nonlinearity
  \textbf{7} no.~6, 1495-1533, 1994.


\bibitem{DAVIES20031787}
H.  C.  Davies and H.~Wernli.
\newblock Quasi-geostrophic theory.
\newblock In James~R. Holton, editor, {\em Encyclopedia of Atmospheric
  Sciences}, pages 1787--1794. Academic Press, Oxford, 2003.




\bibitem{weakvorticityFlandoli}
F.  Flandoli.
\newblock Weak vorticity formulation of 2D Euler equations with white noise
  initial condition.
\newblock {\em Communications in Partial Differential Equations},
  43(7):1102--1149, 2018.

\bibitem{FS} F. Flandoli, M. Saal, mSQG equations in distributional spaces and point vortex approximation. \emph{J. Evol. Equ.} \textbf{19} no. 4, 1071--1090,2019.



\bibitem{Garcia_2020}
C.~Garcia.
\newblock Vortex patches choreography for active scalar equations.
\newblock {\em Journal of Nonlinear Science}, {\bf  31:75} 2021.






\bibitem{Geldhauser_Romito_2020}
C.~Geldhauser and M.~Romito.
\newblock Point vortices for inviscid generalized surface quasi-geostrophic
  models.
\newblock {\em Am. Ins. Math. Sci.}, {\bf 25(7)} 2583--2606.


\bibitem{HPGS} I.M. Held, R.T. Pierrehumbert, S.T. Garner and K.L. Swanson, Surface quasi-geostrophic dynamics. \emph{J. Fluid Mech.} \textbf{282}  1-20,1995.


 \bibitem{Hunter2019}
 J. Hunter,  J. Shu and Q. Zhang.
\newblock{Two-Front Solutions of the SQG Equation and its Generalizations}
\newblock{\em  Communications in Mathematical Sciences  },
18(6),   1685-1741,  2020.



 \bibitem{Hunter2020GlobalSF}
J. Hunter,  J. Shu and Q. Zhang.
\newblock{Global solutions for a family of GSQG front equations}.
\newblock{\em arXiv:2005.09154},    2020.

 \bibitem{Jakubowski1998Short}
A. Jakubowski, {Short Communication:The Almost Sure Skorokhod Representation for Subsequences in Nonmetric Spaces},
 {\it Theory of Probability \& Its Applications},
  {\bf 42} 
  {209-216}, 1998.



\bibitem{LANDAU1987x}
L.D. Landau and E.M. Lifshitz.
\newblock  {\em Fluid Mechanics (Second Edition)}. Pergamon, Second 
  edition, 1987.






\bibitem{lapeyre2017surface}
G. Lapeyre.
\newblock Surface quasi-geostrophy.
\newblock {\em Fluids}, 2(1):7, 2017.


\bibitem {Lions}P.-L. Lions, \textit{Mathematical Topics in Fluid Mechanics},
volume 1, Incompressible Models, Science Publ., Oxford, 1996.

\bibitem{Luo2019RegularizationBN}
D.  Luo and M.  Saal.
\newblock Regularization by noise for the point vortex model of mSQG equations.
\newblock {\em Acta Mathematica Sinica,  English Series}, {\bf 37}: 408-422,  2021.

  




\bibitem{LUO2021236}
D. Luo and R. Zhu,
\newblock Stochastic mSQG equations with multiplicative transport noises: white
  noise solutions and scaling limit, 
  {\it Stochastic Processes and their Applications},
  {\bf 140},  236-286,
  2021.







\bibitem{Marchand2008}
{\sc F.~Marchand}, {\em Existence and regularity of weak solutions to the
  quasi-geostrophic equations in the spaces {$L^p$ or $\dot{H}^{1/2}$}},
  Communications in Mathematical Physics, {\bf 277} , 45-67,2008.


\bibitem{10.1214/16-AOP1116}
J. Mourrat and H. Weber,
\newblock {Global well-posedness of the dynamic $\Phi^{4}$ model in the plane}.
\newblock {\em The Annals of Probability},  {\bf 45(4)}  2398-2476, 2017.

\bibitem{PARTHASARATHY1967v}
K.R. Parthasarathy,
\newblock {\em Probability Measures on Metric Spaces}.
\newblock Academic Press, 1967.


\bibitem{ped} J. Pedlovsky, {\it Geophysical Fluid Dynamics,} Springer
1979.








\bibitem{Res95}
{ S.~Resnick}, { Dyanmical Problems in Non-Linear Advective Partial
  Differential Equations}, PhD thesis, University of Chicago, 1995.

 \bibitem{rosenzweig} M.  Rosenzweig, { Justification of the point vortex approximation for modified surface quasigeostrophic equations}, {\it SIAM J. Math. Anal. }, {\bf 52} no. 2, 1690-1728,2020.
 
\bibitem{bookHans}
H.  Triebel.
\newblock {\em Theory of Function Spaces III}.
\newblock 01 2006.





\bibitem{vallis2017atmospheric}
G~K Vallis.
\newblock {\em Atmospheric and oceanic fluid dynamics}.
\newblock Cambridge University Press, 2017.	







\end{thebibliography}

\end{document}